\documentclass[a4paper,11pt]{article}

\usepackage[leqno]{amsmath}
\usepackage{amssymb,amsthm,ifthen}
\usepackage{url,enumitem}
\usepackage[pdftex]{graphicx}
\usepackage[curve,knot]{xypic}

\usepackage{yhmath}

\usepackage{tikz}
\usetikzlibrary{calc}
\usepgflibrary{shapes.geometric}
\usepgflibrary{shapes.misc}
\usetikzlibrary{positioning}
\usetikzlibrary{decorations}
\usetikzlibrary{arrows}

\usepackage[a4paper,left=1in,width=458pt,textheight=650pt,top=1.5in]{geometry}

\usepackage{bbm}

\usepackage{hyperref}

\input cyracc.def

\hypersetup{
pdftitle={Graded quantum cluster algebras and an application to quantum Grassmannians},
pdfauthor={Jan E. Grabowski and Stephane Launois},
pdfstartview=FitH
}

\newcommand{\bigdsum}{\bigoplus}

\newcommand{\card}[1]{\lvert #1 \rvert}

\let\chisave\chi
\renewcommand{\chi}{{%
 \mathchoice{\raisebox{0.25ex}{$\displaystyle\chisave$}}
            {\raisebox{0.2ex}{$\textstyle\chisave$}}
            {\raisebox{0.2ex}{$\scriptstyle\chisave$}}
            {\raisebox{0.1ex}{$\scriptscriptstyle\chisave$}}}}

\newcommand{\complex}{\ensuremath \mathbb{C}}

\newcommand{\cross}{\times}
\newcommand{\curly}[1]{\ensuremath{\mathcal{#1}}}
\newcommand{\disjointunion}{\sqcup}

\newcommand{\defeq}{\stackrel{\scriptscriptstyle{\mathrm{def}}}{=}}

\newcommand{\dual}[1]{\ensuremath {#1}^{*}}

\renewcommand{\epsilon}{\varepsilon}

\newcommand{\inj}{\hookrightarrow}
\newcommand{\integ}{\ensuremath{\mathbb{Z}}}
\newcommand{\intersection}{\mathrel{\cap}}

\newcommand{\iso}{\ensuremath \cong}

\newcommand{\KqGr}[2]{\mathbb{K}_{q}[\mathrm{Gr}(#1,#2)]}
\newcommand{\KqMat}[2]{\mathbb{K}_{q}[\mathrm{M}(#1,#2)]}

\newcommand{\minor}[2]{\genfrac{[}{]}{0em}{1}{#2}{#1}}

\newcommand{\nat}{\ensuremath \mathbb{N}}

\renewcommand{\phi}{\varphi}

\newcommand{\transpose}[1]{{#1}^{T}}

\newcommand{\union}{\mathrel{\cup}}

\newcommand{\ie}{i.e.\ }
\newcommand{\eg}{e.g.\ }

\theoremstyle{plain}
\newtheorem{theorem}{Theorem}[section]
\newtheorem*{theorem*}{Theorem}
\newtheorem{proposition}[theorem]{Proposition}
\newtheorem{lemma}[theorem]{Lemma}
\newtheorem{corollary}[theorem]{Corollary}

\theoremstyle{definition}
\newtheorem{definition}[theorem]{Definition}
\theoremstyle{remark}
\newtheorem{remark}[theorem]{Remark}

\newtheorem*{example*}{Example}
\newtheorem*{examplectd*}{Example (continued)}

\title{Graded quantum cluster algebras \\ and an application to quantum Grassmannians}
\author{Jan E. Grabowski\footnotemark[2] 
\\ \small{\textit{Department of Mathematics and Statistics, Lancaster University,}}
\\ \small{\textit{Lancaster, LA1 4YF, United Kingdom}}
\and St\'{e}phane Launois\footnotemark[3]
\\ \small{\textit{School of Mathematics, Statistics and Actuarial Science,
University of Kent,}}
\\ \small{\textit{Canterbury, CT2 7NF, United Kingdom}}
}
\date{10th January 2013}

\begin{document}

\maketitle

\renewcommand{\thefootnote}{\fnsymbol{footnote}}
\footnotetext[2]{Email: \url{j.grabowski@lancaster.ac.uk}.  Website: \url{http://www.maths.lancs.ac.uk/~grabowsj/}}
\footnotetext[3]{Email: \url{S.Launois@kent.ac.uk}.  Website: \url{http://www.kent.ac.uk/smsas/personal/sl261/}}
\renewcommand{\thefootnote}{\arabic{footnote}}
\setcounter{footnote}{0}

\begin{abstract}
We introduce a framework for $\integ$-gradings on cluster algebras (and their quantum analogues) that are compatible with mutation.  To do this, one chooses the degrees of the (quantum) cluster variables in an initial seed subject to a compatibility with the initial exchange matrix, and then one extends this to all cluster variables by mutation.  The resulting grading has the property that every (quantum) cluster variable is homogeneous.

In the quantum setting, we use this grading framework to give a construction that behaves somewhat like twisting, in that it produces a new quantum cluster algebra with the same cluster combinatorics but with different quasi-commutation relations between the cluster variables.

We apply these results to show that the quantum Grassmannians $\KqGr{k}{n}$ admit quantum cluster algebra structures, as quantizations of the cluster algebra structures on the classical Grassmannian coordinate ring found by Scott.   This is done by lifting the quantum cluster algebra structure on quantum matrices due to Gei\ss--Leclerc--Schr\"{o}er and completes earlier work of the authors on the finite-type cases.  
\end{abstract}


\vfill

\pagebreak

\section{Introduction}

Since cluster algebras were introduced by Fomin and Zelevinsky (\cite{FZ-CA1}), it has been recognised that cluster algebra structures on homogeneous coordinate rings on Grassmannians are among the most important classes of examples.  The demonstration of the existence of such a structure is due to Scott \cite{Scott-Grassmannians} and one reason for the importance of this is that these cluster structures are typically of infinite mutation type but have combinatorics under tight control, because they are a realisation of certain aspects of Grassmannian combinatorics.  We note for example that Fomin and Pylyavskyy (\cite{FominPylyavskyy}) have recently advocated further study of Grassmannian cluster structures for precisely these reasons.

Among those who study quantized coordinate rings, it is also widely acknowledged that Grassmannians have a special place.  Again, the intricate geometric structures associated to Grassmannians, due in part to their Lie-theoretic origins, give a rich structure of their quantized coordinate rings, the quantum Grassmannians $\KqGr{k}{n}$.  Linking these two points of view, it has long been expected that quantum Grassmannians should possess quantum cluster algebra structures, the definition of the latter being due to Berenstein and Zelevinsky (\cite{BZ-QCA}).  In earlier work, the present authors showed that this is the case when the cluster structure was expected to be of finite type, namely the cases $\KqGr{2}{n}$ and $\KqGr{3}{n}$ for $n=6,7,8$.  However a general proof was not given at that time: one aim of this paper was to give such a general proof and this is achieved in Theorem~\ref{t:quotisotoGr}.

In the course of attempting to generalise our earlier work on quantum cluster algebra structures, it became apparent that new techniques would be required to handle the general case.  The main tools needed were ways to transfer quantum cluster algebra structures between related algebras.  In the commutative setting, many of these operations are easy to carry out but the noncommutative situation is considerably more delicate, as one must be sure not to destroy the property of quantum clusters consisting of pairwise quasi-commuting elements.  (That is, variables in the same quantum cluster should commute up to a power of the deformation parameter $q$.)

Examination of these problems showed that the correct way to keep control of this problem in the quantum setting is, as often in quantum groups, to introduce gradings.  While the definition and constructions here were originally formed with a view to solving the quantum Grassmannian problem, we believe that the framework we introduce here should be of significance to researchers interested in cluster algebras more generally.  In particular, this framework applies to classical commutative cluster algebras as well as their quantum analogues and yields in a natural way statements about the homogeneity of (quantum) cluster variables in a variety of settings.  We remark that these results sit among a surprisingly small number that deal with cluster algebras and properties of their cluster variables in infinite types alongside finite types.  The definition of a graded cluster algebra also sits separately from the categorical setting (i.e.\ cluster categories and related constructions), though---as here---it is fully compatible with categorification and indeed the two can illuminate each other.

We therefore devote the first two main sections of this work to graded (quantum) cluster algebras, following some essential recollections.  The first gives the definition of a graded cluster algebra: it is obtained from an initial seed by adding an additional piece of data, an assignment of integer degrees to each initial cluster variable, subject to a compatibility condition with the initial exchange matrix.  (This compatibility is only with the exchange matrix, which is why the notion immediately extends to the quantum setting.)  This initial data is propagated to the whole cluster algebra by mutation, the key point being that we can mutate the grading data in a natural way.  An immediate consequence of the definition is that every (quantum) cluster variable is homogeneous for the resulting algebra grading.  We note that Berenstein and Zelevinsky have introduced a similar notion of grading in \cite{BZ-QCA}.  However this was not an algebra grading but only a module grading.  Their definition helped inspire ours but the two are different.

We then give a number of constructions that use the grading to alter a given quantum cluster algebra structure; some of these constructions are trivial for commutative cluster algebras.  For example, one may re-scale every initial quantum cluster variable for a graded quantum cluster algebra by $q^{1/2}$ to the power of its degree and obtain an isomorphic quantum cluster algebra.  We also show how to naturally extend a quantum cluster algebra to a skew-Laurent extension of that algebra.  Furthermore we can combine these ideas to ``re-scale'' a quantum cluster algebra structure using a skew-Laurent extension, namely Theorem~\ref{t:rescaledQCA}.  This theorem is a twisting-like result, in that the new quantum cluster algebra structure so obtained has the same cluster combinatorics but different quasi-commutation relations.  The existence of gradings is key for that result and in turn Theorem~\ref{t:rescaledQCA} is key to the application we describe below, to the quantum Grassmannian.

The problem of lifting the cluster algebra structure on the coordinate ring of the big cell of a partial flag variety has been solved by Gei\ss, Leclerc and Schr\"{o}er in \cite[\S 10]{GLS-PFV}.  The approach is straightforward: one can view the coordinate ring of the cell as the quotient of the coordinate ring of the partial flag variety by certain elements (minors), and this quotient allows one to lift certain distinguished elements from the quotient to homogeneous elements in a minimal (and hence unique) way.  Then \cite[Proposition~10.1]{GLS-PFV} shows that every cluster variable has the required property and hence the lifting procedure for the whole cluster algebra structure is possible.  Recalling that the cluster algebra structure on the coordinate ring of the cell is obtained categorically, one notes that the appropriate data for the lifting is also encoded categorically, giving rise to a hope that this may be used in the quantum setting also.

However it is not possible to follow the strategy of \cite{GLS-PFV} directly in the quantum setting.  It is well-known that in the noncommutative setting, factors of rings by normal but not central elements can be ``too small'' and so the appropriate construction is localisation, in the form of noncommutative dehomogenisation (see \cite[\S 3]{KellyLenaganRigal}, for example).  Therefore our methods are necessarily different to, and indeed more technically complicated than, the approach of \cite{GLS-PFV}.

This necessitated the introduction of the graded methods described above.  As an application of these, we are able to give a noncommutative version of the lifting of Gei\ss--Leclerc--Schr\"{o}er and prove that the quantum Grassmannians $\KqGr{k}{n}$ admit graded quantum cluster algebra structures.  The main prior results used are the existence of a quantum cluster algebra structure on $\KqMat{k}{n-k}$, as shown by Gei\ss--Leclerc--Schr\"{o}er (\cite{GLS-QuantumPFV}), and the dehomogenisation isomorphism, due originally to Kelly--Lenagan--Rigal (\cite{KellyLenaganRigal}), though we use a version given by Lenagan--Russell (\cite{LenaganRussell}).  The dehomogenisation isomorphism makes the key link between quantum matrices and the quantum Grassmannian and is the noncommutative expression of the former as the quantum coordinate ring of the big cell of the latter.  That is, it is an isomorphism 
\[ \alpha\colon \KqMat{k}{n-k}[Y^{\pm 1}; \sigma] \to \KqGr{k}{n}\!\left[ [12\cdots k]^{-1} \right] \]
between a certain skew-Laurent extension of the quantum matrices and a certain localisation of the quantum Grassmannian.  The main goal is to transfer a quantum cluster algebra structure through this isomorphism $\alpha$ and show that it can be lifted from the localisation to the quantum Grassmannian $\KqGr{k}{n}$ itself.

\vfill
\pagebreak
As this construction is rather technical, for the benefit of the reader we give a detailed breakdown of the structure of the proof, as follows:
\renewcommand{\labelenumi}{\textbf{(\Alph*)}}
\renewcommand{\labelenumii}{\textbf{(\Alph{enumi}\arabic*)}}
\renewcommand{\labelenumiii}{\textbf{(\Alph{enumi}\arabic{enumii}\alph*)}}
\begin{enumerate}
\item Analysis of the quantum cluster algebra structure on $\KqMat{k}{n-k}$ given by \cite{GLS-QuantumPFV}. \hfill (\S \ref{s:KqMatisQCA})
\begin{enumerate}
\item Proof of the existence of a \emph{graded} quantum cluster algebra structure on $\KqMat{k}{n-k}$. \rule{0em}{1em} \hfill (Lemma~\ref{lemma:degsumsequal})
\item Observation of the existence of an ``almost-grading'' $\theta$, arising from the categorification of $\KqMat{k}{n-k}$. \hfill (p.\ \pageref{page:theta-start}--\pageref{page:theta-end})
\end{enumerate}
\item Identification of the image of the quantum cluster algebra structure on a skew-Laurent extension of $\KqMat{k}{n-k}$ under the dehomogenisation isomorphism $\alpha$ described above, giving a quantum cluster algebra structure on a localisation of $\KqGr{k}{n}$. \hfill (\S \ref{s:dehomog})
\begin{enumerate} 
\item\label{B1} Identification of the images of quantum cluster variables. \hfill (Theorem~\ref{t:thetaisscalingpower})
\item\label{B2} In particular, identification of the images of quantum minors. \hfill (Corollary~\ref{c:PluckersAreQCVs})
\end{enumerate}
\item ``Re-homogenisation'' by transferring of the quantum cluster algebra structure from the localisation to the unlocalised algebra. \hfill (\S \ref{s:QCAonKqGr})
\begin{enumerate}
\item Alteration of the quantum cluster algebra structure on the localisation of $\KqGr{k}{n}$ such that the ``almost-grading'' $\theta$ becomes an honest grading. \hfill (p.\ \pageref{page:fix-theta-start}--\pageref{page:fix-theta-end})
\item Use of Theorem~\ref{t:rescaledQCA} to produce a new quantum cluster algebra with the same cluster combinatorics but whose quasi-commuting relations now match those of the quantum Grassmannian.  The quantum cluster variables of the new algebra are shown to be products of elements of $\KqGr{k}{n}$ multiplied by a power of a certain central element (that power being controlled by $\theta$). \hfill (Proposition~\ref{p:rescaledLoctilde}, Lemma~\ref{l:centralelt})
\item The quotient that sets the above central element to 1 inherits a quantum cluster algebra structure (with the same cluster combinatorics). \hfill (Corollary~\ref{c:quotisQCA})
\item Demonstration that this quotient is isomorphic to $\KqGr{k}{n}$. \hfill (Theorem~\ref{t:quotisotoGr})
\begin{enumerate}
\item \textbf{(B1)}, \textbf{(B2)} and \textbf{(C2)} above collectively imply that there exists a surjective homo\-morphism. 
\item The equalities of the Gel\cprime fand--Kirillov dimensions of the two algebras shows that this epimorphism is an isomorphism.
\end{enumerate}
\item Finally, the powers of $q$ appearing in the expressions for the quantum cluster variables can be removed, as another consequence of the grading. \hfill (after Theorem~\ref{t:quotisotoGr})
\end{enumerate}
\end{enumerate}

This construction does indeed yield the same results as the authors' earlier work (\cite{Gr2nSchubertQCA}) but now in arbitrary---and in particular far from finite---type, in a universal way.  Other consequences of the approach taken here include the fact that every quantum cluster variable for this structure on $\KqGr{k}{n}$ is homogeneous for the standard grading, with the quantum Pl\"{u}cker coordinates in degree one.  It is hard to see how this property could be established in infinite type without the framework presented here.  We note that we do not deduce this from any explicit formul\ae; indeed, we have no such formul\ae, though it would be interesting to understand the forms of the quantum cluster variables in tame types ($\KqGr{3}{9}$ and $\KqGr{4}{8}$), now that the existence of the quantum cluster algebra structure is proved.

We also note that the above proof does not in fact rely on any particularly special properties of the quantum Grassmannian.  Indeed, many of the steps outlined above have analogues for quantum coordinate rings of big cells of partial flag varieties in full generality; by the latter, we mean the algebras $A_{q}(\mathfrak{n}(w_{0}w_{0}^{K}))$ discussed in \cite[\S 12.5]{GLS-QuantumPFV}.  In particular, Corollary~12.12 of that paper gives us a quantum cluster algebra structure to take as ``input'' to the process described above.  Similarly, corresponding dehomogenisation isomorphisms in this more general case are known (\cite{Yakimov}, \cite{Yakimov-2}).  

However some adjustments may be needed in the general case.  From the quantum Grassmannian $\KqGr{k}{n}$ we localise by the Ore set $\{ [1\dotsm k]^{n} \mid n\in \nat \}$, which naturally gives the localisation a $\integ$-grading.  Then we obtain $\KqMat{k}{n-k}$ as the degree zero part of this.  In general, we will need to localise by powers of several elements, leading to multi-gradings on the localisation.  Hence we will need to work with multi-gradings on the quantum cluster algebras $A_{q}(\mathfrak{n}(w_{0}w_{0}^{K}))$, to lift that structure to the quantized coordinate ring of the partial flag variety itself.  We intend to return to this topic in more detail in future work.

\subsection*{Acknowledgements} The second author is grateful for the financial support of EPSRC First Grant \emph{EP/I018549/1}.

\section{Preliminaries}

\subsection{Quantum matrices and quantum Grassmannians}\label{ss:qGrassmannians}

Throughout, $\mathbb{C}$ denotes the field of complex numbers and $\mathbb{K}$ denotes the field $\mathbb{Q}(q^{1/2})$. Then in particular, the indeterminate $q$ is not a root of unity and has a square root in $\mathbb{K}$. Let $C$ be an $l\cross l$ generalized Cartan matrix with columns indexed by a set $I$.  Let $(H,\Pi,\Pi^{\vee})$ be a minimal realization of $C$, where $H \iso \complex^{2\left| I \right|-\text{rank}(C)}$, $\Pi=\{ \alpha_{i} \mid i \in I \} \subset \dual{H}$ (the simple roots) and $\Pi^{\vee}=\{ h_{i} \mid i \in I \} \subset H$ (the simple coroots).  Then we say $\curly{C}=(C,I,H,\Pi,\Pi^{\vee})$ is a root datum associated to $C$.  (Lusztig \cite{LusztigBook} has a more general definition of a root datum but this one will suffice for our purposes.)

If $G=G(\curly{C})$ is a connected semisimple complex algebraic group associated to $\curly{C}$, $G$ has a (standard) parabolic subgroup $P_{J}$ associated to any choice of subset $J\subseteq I$.  From this we can form $G/P_{J}$, a partial flag variety; the choice $J=\emptyset$ gives $G/P_{\emptyset}=G/B$, the full flag variety.  We set $D=I\setminus J$.

The partial flag variety $G/P_{J}$ is a projective variety, via the well-known Pl\"{u}cker embedding $G/P_{J} \inj \prod_{d\in D} \mathbb{P}(L(\omega_{d}))$.  (Here, $L(\lambda)$ is the irreducible $G$-module corresponding to a dominant integral weight $\lambda$ and $\{\omega_{i}\}_{i\in I}$ are the fundamental weights.)  Via the Pl\"{u}cker embedding, we may form the corresponding $\nat^{D}$-graded multi-homogeneous coordinate algebra $\complex[G/P_{J}]=\bigoplus_{\lambda \in \nat^{D}} \dual{L(\lambda)}$.  The case we consider is that of the partial flag variety obtained from $G=G(A_{n})=SL_{n+1}(\complex)$ and $J=I \setminus \{k\}$, namely $G/P_{J}=\mathrm{Gr}(k,n)$, the Grassmannian of $k$-dimensional subspaces in $\complex^{n}$.   

The coordinate ring $\complex[G]$ has a quantum analogue, $\mathbb{K}_{q}[G]$ (see for example \cite{Brown-Goodearl}, where this algebra is denoted $\curly{O}_{q}(G)$).  Via this quantized coordinate ring, we can define a quantization $\mathbb{K}_{q}[G/P_{J}]$ of $\mathbb{K}[G/P_{J}]$.  

We recall that the quantum matrix algebra $\KqMat{k}{n}$ is the $\mathbb{K}$-algebra generated by the set $\{ X_{ij} \mid 1\leq i\leq k,\ 1\leq j \leq n \}$ subject to the quantum $2\cross 2$ matrix relations on each $2\cross 2$ submatrix of \[ \begin{pmatrix} X_{11} & X_{12} & \cdots & X_{1n} \\ \vdots & \vdots & \ddots & \vdots \\ X_{k1} & X_{k2} & \cdots & X_{kn} \end{pmatrix}. \] The quantum $2\cross 2$ matrix relations on $\left( \begin{smallmatrix} a & b \\ c & d \end{smallmatrix} \right)$ are
\begin{align*} ab & = qba & ac & = qca & bc & = cb \\ bd & = qdb & cd & = qdc & ad-da & = (q-q^{-1})bc. \end{align*}
 
Recall that the $k\cross k$ quantum minor $\Delta_{q}^{I}$ associated to the $k$-subset $I=\{ i_{1} < i_{2} < \cdots < i_{k} \}$ of $\{1, \dots , n\}$ is defined to be
\[ \Delta_{q}^{I} \defeq \sum_{\sigma \in S_{k}} (-q)^{l(\sigma)}X_{1i_{\sigma(1)}}\cdots X_{k\mspace{0.5mu}i_{\sigma(k)}} \] where $S_{k}$ is the symmetric group of degree $k$ and $l$ is the usual length function on this.  This defines the quantum minor $\Delta_{\{1,\ldots,k\}}^{I}$ but quantum minors of smaller degree or for different choices of the row set are defined analogously in the obvious way.  

Our notation for a quantum minor with row set $I$ and column set $J$ will be $\minor{I}{J}$; note that we suppress $q$, as all our minors will be quantum minors unless otherwise stated.  For example, when $k=2$, we will write the quantum minor $\Delta_{q}^{ij}$ for $i<j$ as $\minor{12}{ij}$; written in terms of the generators of $\KqMat{2}{2}$ this is equal to $X_{1i}X_{2j}-qX_{1j}X_{2i}$.  Similarly, we will often denote the generator $X_{ij}$ of $\KqMat{k}{n}$ by $(ij)$.

Then we denote by $\curly{P}_{q}$ the set of all quantum Pl\"{u}cker coordinates, that is $$\curly{P}_{q} = \{ \Delta_{q}^{I}  \mid I \subseteq \{1,\ldots ,n\}, \card{I}=k \}.$$  

\begin{definition}\label{def:KqGrkn} The quantum Grassmannian $\KqGr{k}{n}$ is the subalgebra of the quantum matrix algebra $\KqMat{k}{n}$ generated by the quantum Pl\"{u}cker coordinates $\curly{P}_{q}$.
\end{definition}

When working with minors in the quantum Grassmannian $\KqGr{k}{n}$, where the row set is necessarily $\{ 1,\ldots ,k\}$, we will write $[I]$ for $\Delta_{q}^{I}$, \eg $[ij]$ for $\Delta_{q}^{ij}$ as above.  

It is well-known that $\KqGr{k}{n}$ is a Noetherian domain with Gel\cprime fand--Kirillov dimension $k(n-k)+1$.  (Further ring-theoretic properties of quantum Grassmannians are established in \cite{Fioresi}, \cite{KellyLenaganRigal}, \cite{LaunoisLenaganRigal} and \cite{LenaganRigal}.)

\subsection{Cluster algebras}

The construction of a cluster algebra of geometric type from an initial seed $(\underline{x},B)$, due to Fomin and Zelevinsky (\cite{FZ-CA1}), is now well-known.  Here $\underline{x}$ is a transcendence base for a certain field of fractions of a polynomial ring and $B$ is a skew-symmetric integer matrix; often $B$ is replaced by the quiver $Q=Q(B)$ it defines in the natural way.  We refer the reader who is unfamiliar with this construction to the survey of Keller (\cite{Keller-CASurvey}) and the recent book of Gekhtman, Shapiro and Vainshtein (\cite{GSV-Book}) for an introduction to the topic and summaries of the main related results in this area.

\subsection{Quantum cluster algebras}

Berenstein and Zelevinsky (\cite{BZ-QCA}) have given a definition of a quantum cluster algebra.  These algebras are non-commutative but not so far from being commutative.  A quantum seed $(\underline{x},B,L)$ consists of $\underline{x}=(X_{1},\dotsc ,X_{r})$, simultaneously a transcendence base and generating set for the skew-field of fractions $\curly{F}$ of a quantum torus (over the field $\mathbb{K}$), a skew-symmetric integer matrix $B$ (the exchange matrix) and a second skew-symmetric integer matrix $L=(l_{ij})$ that determines the aforementioned quantum torus.  That is, the matrix $L$ describes quasi-commutation relations between the variables in the cluster, where quasi-commuting means the existence of a relation of the form $X_{i}X_{j}=q^{l_{ij}}X_{j}X_{i}$.

There is also a mutation rule for these quasi-commutation matrices as well as a modified exchange relation that involves further coefficients that are powers of $q$ derived from $B$ and $L$, which we describe now.  For $k$ a mutable index, set
\begin{align*}
\underline{b}_{k}^{+} & = -\underline{\boldsymbol{e}}_{k}+\sum_{b_{ik}>0}b_{ik}\underline{\boldsymbol{e}}_{i} \qquad \text{and} \\
\underline{b}_{k}^{-} & = -\underline{\boldsymbol{e}}_{k}-\sum_{b_{ik}<0}b_{ik}\underline{\boldsymbol{e}}_{i}
\end{align*}
where the vector $\underline{\boldsymbol{e}}_{i}\in \complex^{r}$ ($r$ being the number of rows of $B$) is the $i$th standard basis vector.  Note that the $k$th row of $B$ may be recovered as $B_{k}=\underline{b}_{k}^{+}-\underline{b}_{k}^{-}$.

Then given a quantum cluster $\underline{x}=(X_{1},\ldots,X_{r})$, exchange matrix $B$ and quasi-commutation matrix $L$, the exchange relation for mutation in the direction $k$ is given by
\[ X_{k}^{\prime}=M(\underline{b}_{k}^{+})+M(\underline{b}_{k}^{-}) \]
with
\[ M(a_{1},\dotsc ,a_{r}) \defeq q^{\frac{1}{2}\sum_{i<j} a_{i}a_{j}l_{ji}}X_{1}^{a_{1}}\dotsm X_{r}^{a_{r}}. \]
By construction, the integers $a_{i}$ are all non-negative except for $a_{k}=-1$.  The monomial $M$ (as we have defined it here) is related to the concept of a toric frame, also introduced in \cite{BZ-QCA}.  The latter is a technical device used to make the general definition of a quantum cluster algebra.  For our examples, where we start with a known algebra and want to exhibit a quantum cluster algebra structure on this, it will suffice to think of $M$ simply as a rule determining the exchange monomials.

The quantum cluster algebra $\curly{A}_{q}=\curly{A}_{q}(\underline{x},B,L)$ defined by the initial data $(\underline{x},B,L)$ is the subalgebra of $\curly{F}$ generated by the set of all quantum cluster variables, that is, those elements of $\curly{F}$ obtained from the initial cluster variables by iterated mutation.  We note that the presence of the factor $1/2$ in the quantum exchange relations is the reason for assuming that the element $q\in \mathbb{K}$ has a square root.

We will need to work with quantum cluster algebras with coefficients (also called frozen variables).  That is, we designate some of the elements of the initial cluster to be mutable (i.e. we are allowed to mutate these) and the remainder to be non-mutable.  We will also talk about the corresponding indices for the variables as being mutable or not; in \cite{BZ-QCA} the former are referred to as exchangeable indices.  The rank of the quantum cluster algebra is the number of mutable variables in a cluster; we will refer to the total number of variables, mutable and not, as the cardinality of the cluster.

Note that we will adopt the convention that $B$ will be a \emph{square} matrix---in the literature it is more common to let $B$ have as column indices just the mutable indices (\cite{GSV-Book} adopts the transpose convention, of having the rows indexed by the mutable indices).  At some points---notably in the next paragraph---we will need the submatrix $B_{\text{mut}}$ of $B$ given by taking only the columns of $B$ with mutable indices.  The matrix $B_{\text{mut}}$ is often referred to as the extended exchange matrix and \emph{its} submatrix $B_{\text{mut}}^{\textrm{mut}}$ with row set also the mutable indices is what is usually called the principal part of $B$.  Our square matrix $B$ is simply the ``skew-symmetric extension'' of $B_{\text{mut}}$, i.e. completing $B_{\text{mut}}$ to a square matrix in the unique way so that $B$ is skew-symmetric and so that if $i$ and $j$ are non-mutable indices then $b_{ij}=0$.  (The latter choice accords with the convention that the exchange quiver has no arrows between frozen vertices.)

The natural requirement that all mutated quantum clusters also quasi-commute leads to a compatibility condition between $B$ and $L$, namely that $\transpose{(B_{\text{mut}})}L$ consists of two blocks, one a positive integer scalar multiple of the identity and one zero.  The non-zero block must correspond exactly to the mutable (column) indices.  However, these blocks need not be contiguous, depending on the ordering of the row and column labels.  (Only one positive integer is required, as the principal part of $B$ is assumed to be skew-symmetric; if one assumes just skew-symmetrizability then the non-zero block is only required to be diagonal with positive integer diagonal entries.)  

Let us say that an arbitrary skew-symmetric matrix $A$ is $d$-compatible with $B$ if $\transpose{(B_{\text{mut}})}A$ has the form described above for some non-negative integer $d$.  It will suit our purposes later to allow $d=0$ (i.e. $\transpose{(B_{\text{mut}})}A=0$), even though $0$-compatibility is not permitted for the compatibility of a quasi-commutation matrix $L$ with $B$.

Importantly, Berenstein and Zelevinsky show that the exchange graph (whose vertices are the clusters and edges are mutations) remains unchanged in the quantum setting.  That is, the matrix $L$ does not influence the exchange graph.  It follows that quantum cluster algebras of finite type are classified by Dynkin types in exactly the same way as the classical cluster algebras.

Known or conjectured examples of quantum cluster algebras include
\begin{itemize}
\item quantum symmetric algebras (necessarily of cluster algebra rank 0);
\item quantum Grassmannians of finite cluster algebra type (\ie $\KqGr{2}{n}$ and $\KqGr{3}{6}$, $\KqGr{3}{7}$ and $\KqGr{3}{8}$) (\cite{Gr2nSchubertQCA});
\item Schubert cells of the quantum Grassmannians $\KqGr{2}{n}$ (\cite{Gr2nSchubertQCA},\cite{GLS-QuantumPFV});
\item the quantum coordinate ring of the unipotent subgroup $N(w)$ of a symmetric Kac--Moody group $G$ associated with a Weyl group element $w$ (\cite{GLS-QuantumPFV}), and hence as special cases of this
\begin{itemize}
\item the quantum coordinate ring of the big cell of a partial flag variety associated to $G$ and 
\item quantum matrices $\KqMat{k}{n}$; and
\end{itemize}
\item conjecturally, quantum double Bruhat cells of semisimple Lie groups (\cite{BZ-QCA}). 
\end{itemize}
We note that recently Goodearl and Yakimov (\cite{GoodearlYakimov}) have studied the existence of initial seeds in a large class of algebras, the so-called CGL extensions.

\section{Graded seeds and graded quantum cluster algebras}\label{s:gradedQCAs}

Berenstein and Zelevinsky (\cite[Definition~6.5]{BZ-QCA}) have given a definition of graded quantum seeds, which give rise to module gradings but not algebra gradings.  In what follows, we will have need of algebra gradings on quantum cluster algebras and so we now give a different definition of a graded seed, inspired by that of Berenstein and Zelevinsky but not equivalent to it.

\begin{definition} A graded quantum seed is a quadruple $(\underline{x},B,L,G)$ such that
\begin{enumerate}[label=(\alph*)]
\item $(\underline{x}=(X_{1},\dotsc ,X_{r}),B,L)$ is a quantum seed of cardinality $r$ and
\item $G\in \integ^{r}$ is an integer (column) vector such that for all mutable indices $j$, the $j$th row of $B$, $B_{j}$, satisfies $B_{j}G=0$.
\end{enumerate}
\end{definition}

We will set $\deg_{G}(X_{i})=G_{i}$ for all $X_{i}$ belonging to the cluster $\underline{x}$.  Then the second condition of the definition is equivalent to asking that every exchange relation (as encoded by the rows $B_{j}$) is homogeneous with respect to this degree, in the sense that the two Laurent monomials determining $X_{j}^{\prime}$ are of the same homogeneous degree.  From the quiver perspective, this asks that the sum of the degrees of the variables with arrows to a given mutable vertex is equal to the sum of the degrees of the variables at the end of arrows leaving that vertex.

In contrast to the definition of Berenstein and Zelevinsky, the above can be extended to an algebra grading on the quantum torus associated to $(\underline{x},B,L)$, simply by setting $\deg_{G}(X_{i}^{-1})=-\deg_{G}(X_{i})$ and extending $\deg_{G}$ additively to all (Laurent) monomials.

We also need to be able to mutate our grading in a sensible fashion and it is clear what we ought to do.  Let $(\underline{x}^{\prime},B^{\prime},L^{\prime})$ be the quantum seed given by mutation of $(\underline{x},B,L)$ in the direction $j$.  We set $G^{\prime}_{i}=G_{i}$ for $i\neq j$ (i.e. the degrees of variables we are not mutating at this point remain the same).  Then the homogeneity of the exchange relation $X_{j}^{\prime}=M(\underline{b}_{k}^{+})+M(\underline{b}_{k}^{-})$ implies that we should set 
\[ G^{\prime}_{j} =\deg_{G^{\prime}}(X_{j}^{\prime}) =\deg_{G}(M(\underline{b}_{k}^{+})) = \deg_{G}(M(\underline{b}_{k}^{-})). \]

As discussed in \cite{BFZ-CA3} and \cite{BZ-QCA}, the mutation operations can also be expressed in terms of row and column operations, or more concisely as corresponding matrix multiplications.  To this end, we recall the definition of a matrix $E$ (denoted $E_{+}$ in \cite{BZ-QCA}) that encodes mutation of a seed with exchange matrix $B$ in the direction $j$ as follows:
\[ E_{rs}=
\begin{cases} 
\delta_{rs} & \text{if}\ s\neq j; \\
-1 & \text{if}\ r=s=j; \\
\max(0,-b_{rj}) & \text{if}\ r\neq s=j.
\end{cases}
\]
Then $B^{\prime}=EB\transpose{E}$ and $L^{\prime}=\transpose{E}LE$.  Our mutation of $G$ can be written in terms of $E$ similarly.

\begin{lemma}\label{l:mutofgradinggivenbyE} $G^{\prime}=\transpose{E}G$.
\end{lemma}

\begin{proof} This is straightforward to check directly.
\end{proof}

We may also re-express this in terms of the vectors $\underline{b}_{k}^{\pm}$ defined above:
\[ G_{i}^{\prime} = \begin{cases} G_{i} & \text{if}\ i\neq k \\ \underline{b}_{k}^{-}\cdot G & \text{if}\ i=k \end{cases}. \]
Since $B_{k}=\underline{b}_{k}^{+}-\underline{b}_{k}^{-}$ and $G$ is a grading, so $B_{k}G=0$, we have $\underline{b}_{k}^{+}\cdot G=\underline{b}_{k}^{-}\cdot G$ so that we may use $\underline{b}_{k}^{+}$ instead of $\underline{b}_{k}^{-}$ in calculating $G^{\prime}$.

We also need to know that this mutation operation does indeed produce another graded seed.

\begin{lemma}\label{l:mutofgrading} For each mutable index $j$, $(B^{\prime})_{j}G^{\prime}=0$.
\end{lemma}

\begin{proof} As noted in \cite[Proposition~3.4]{BZ-QCA}, $E^{2}=1$ and so
\begin{align*}
(B^{\prime})_{j}G^{\prime} & = (EB\transpose{E})_{j}(\transpose{E}G) \\
 & = (EB(\transpose{E})^{2}G)_{j} \\
 & = (EBG)_{j} \\
 & = E_{j}(BG) \\
 & = 0 
\end{align*}
since $(BG)_{j}=B_{j}G=0$.  Here, $(\ )_{j}$ refers to the $j$th row for matrices and column vectors as appropriate.
\end{proof}

That is, $(\underline{x}^{\prime},B^{\prime},L^{\prime},G^{\prime})$ is again a graded seed.  Furthermore, this mutation is involutive (cf.\ \cite[Proposition~4.10]{BZ-QCA}).  Then we see that repeated mutation propagates a grading on an initial seed to every quantum cluster variable and hence to the associated quantum cluster algebra, as every exchange relation is homogeneous.

\begin{corollary} The quantum cluster algebra $\curly{A}_{q}(\underline{x},B,L,G)$ associated to an initial graded quantum seed $(\underline{x},B,L,G)$ is a $\integ$-graded algebra. \qed
\end{corollary}

We note in particular that this construction---by definition---says that every quantum cluster variable of a graded quantum cluster algebra is homogeneous for this grading.

\begin{remark} It is clear that all of the above is insensitive to replacing $G$ with $-G$, i.e. reversing the sign of every degree.  Indeed for each graded quantum cluster algebra $\curly{A}_{q}(\underline{x},B,L,G)$ we have an isomorphic graded quantum cluster algebra $\curly{A}_{q}^{-}=\curly{A}_{q}(\underline{x},B,L,-G)$ (where ``isomorphic'' here means as quantum cluster algebras, not just as algebras).
\end{remark}

\begin{remark} In none of the above have we used the quasi-commutation matrix $L$.  Indeed all of the above goes through for classical cluster algebras too.
\end{remark}

One consequence of the existence of a grading for a quantum cluster algebra is that this allow us to re-scale elements of the initial seed by powers of $q$ determined by the grading and obtain an isomorphic quantum cluster algebra, as follows.

\begin{proposition}\label{p:rescaleQCAbypowersofq} Let $\curly{A}_{q}=\curly{A}_{q}(\underline{x}=(X_{1},\dotsc ,X_{r}),B,L,G)$ be a graded quantum cluster algebra.  Let $\tilde{\underline{x}}=(\tilde{X}_{1},\dotsc ,\tilde{X}_{r})$ be defined by $\tilde{X}_{i}=q^{G_{i}/2}X_{i}$.  Then there is an isomorphism of graded quantum cluster algebras between $\tilde{\curly{A}}_{q}=\curly{A}_{q}(\tilde{\underline{x}},B,L,G)$ and $\curly{A}_{q}=\curly{A}_{q}(\underline{x},B,L,G)$.
\end{proposition}

\begin{proof} Firstly note that $\underline{x}$ and $\tilde{\underline{x}}$ determine the same quantum torus and hence the corresponding quantum cluster algebras $\curly{A}_{q}$ and $\tilde{\curly{A}}_{q}$ can be viewed as subalgebras of the same skew-field of fractions $\curly{F}$ of this quantum torus.  We show that if $y$ and $\tilde{y}$ are quantum cluster variables for $\curly{A}_{q}$ and $\tilde{A}_{q}$ respectively that are obtained from their respective initial seeds by the same sequence of mutations, then $\tilde{y}=q^{\deg(y)/2}y$ where $\deg$ is the degree induced by the grading $G$.  

More precisely, let $(\underline{x}=(X_{1},\dotsc ,X_{r}),B,L,G)$ be any graded quantum seed (not necessarily equal to the initial data for $\curly{A}_{q}$; we abuse notation here).  Let $\tilde{\underline{x}}=(q^{G_{1}/2}X_{1},\dotsc ,q^{G_{r}/2}X_{r})$.  We claim that if $X_{k}^{\prime}$ is the variable obtained from $\underline{x}$ by mutation in the direction $k$ then $\tilde{X}_{k}^{\prime}=q^{G_{k}^{\prime}/2}X_{k}^{\prime}$, where $G^{\prime}$ is obtained from $G$ by mutation in the direction $k$.  

Firstly, recall the definition of the exchange monomial $M(a_{1},\dotsc ,a_{r})$ as
\[ M(a_{1},\dotsc ,a_{r}) \defeq q^{\frac{1}{2}\sum_{i<j} a_{i}a_{j}l_{ji}}X_{1}^{a_{1}}\dotsm X_{r}^{a_{r}}. \]
Then letting $\underline{a}=(a_{1},\dotsc ,a_{r})$ and
\[ \tilde{M}(\underline{a})=q^{\frac{1}{2}\sum_{i<j} a_{i}a_{j}l_{ji}}\tilde{X}_{1}^{a_{1}}\dotsm \tilde{X}_{r}^{a_{r}} \]
we have that
\begin{align*} \tilde{M}(\underline{a}) & =q^{\frac{1}{2}\sum_{i<j} a_{i}a_{j}l_{ji}}q^{\frac{1}{2}\sum_{i=1}^{r} a_{i}G_{i}}X_{1}^{a_{1}}\dotsm X_{r}^{a_{r}} \\
 & = q^{\frac{\underline{a}\cdot G}{2}}M(\underline{a}).
\end{align*}

Hence
\begin{align*} \tilde{X}_{k}^{\prime} & =\tilde{M}(\underline{b}_{k}^{+})+\tilde{M}(\underline{b}_{k}^{-}) \\
 & = q^{\frac{\underline{b}_{k}^{+} \cdot G}{2}}M(\underline{b}_{k}^{+})+q^{\frac{\underline{b}_{k}^{-} \cdot G}{2}}M(\underline{b}_{k}^{-}) \\
 & = q^{\frac{\underline{b}_{k}^{+} \cdot G}{2}}\left( M(\underline{b}_{k}^{+})+M(\underline{b}_{k}^{-})\right) \\
 & = q^{G_{k}^{\prime}/2}X_{k}^{\prime}
\end{align*}
as required.  Here we have used that $G$ being a grading for $B$ implies that $\underline{b}_{k}^{+}\cdot G=\underline{b}_{k}^{-}\cdot G$ and the equality of both of these with the $k$th entry of the mutated grading $G^{\prime}$, as shown after Lemma~\ref{l:mutofgradinggivenbyE}.

Now since the initial data for $\curly{A}_{q}$ and $\tilde{\curly{A}}_{q}$ differ only in the choice of initial cluster and, by the above, the subalgebras of $\curly{F}$ generated by the respective sets of quantum cluster variables are equal, it follows that these are isomorphic quantum cluster algebras.
\end{proof}

\noindent We note that this construction does not have a counterpart in the classical setting of commutative cluster algebras.  However it should have a semi-classical counterpart, for cluster algebras with compatible Poisson structures (in the sense of Gekhtman--Shapiro--Vainshtein).

\section{Skew-Laurent extensions of quantum cluster algebras}

In this section, we consider two constructions that produce graded quantum cluster algebra structures on skew-Laurent extensions of a given graded quantum cluster algebra.  The first simply adds the extending variable and its inverse as extra coefficients, while the second ``re-scales'' the original structure by multiplying each quantum cluster variable by a power of the extending variable.  This second construction is similar to that in Proposition~\ref{p:rescaleQCAbypowersofq} but in general produces a new quantum cluster algebra not isomorphic to the first.

The first construction proceeds as follows.

\begin{proposition}\label{p:skewLaurent-extra-coeffs} Let $\curly{A}_{q}=\curly{A}_{q}(\underline{x}=(X_{1},\ldots,X_{r}),B,L,G)$ be a graded quantum cluster algebra.  Let $\sigma\colon \curly{A}_{q} \to \curly{A}_{q}$ be an algebra automorphism such that for each $1\leq i \leq r$, there exists $c_{i}\in \integ$ such that $\sigma(X_{i})=q^{c_{i}}X_{i}$; that is, $\sigma$ acts by multiplication by powers of $q$ on the elements of the initial cluster.  Then the skew-Laurent extension $\curly{A}_{q}[y^{\pm 1};\sigma]$ is a graded quantum cluster algebra, with initial data $(\underline{x}',B',L',G')$ where
\begin{itemize}
\item $\underline{x}'=(X_{1},\ldots,X_{r},y,y^{-1})$, with $y$ and $y^{-1}$ additional coefficients,
\item $Q(B')$ is equal to $Q(B) \disjointunion \boxed{y\rule{0em}{0.9em}} \disjointunion \boxed{y^{-1}\rule{0em}{0.9em}}$,
\item $L'$ is determined by the quasi-commutation data in $L$ and the automorphism $\sigma$, and
\item $(G')_{i}=G_{i}$ for $1\leq i \leq r$, $G'_{r+1}=1$ and $G'_{r+2}=-1$.
\end{itemize}
\end{proposition}

\begin{proof}
The matrices $B'$ and $L'$ are compatible, as $B'$ has only zero entries in the rows corresponding to $y$ and $y^{-1}$.  The matrices $B'$ and $G'$ satisfy the grading condition for the same reason.  Furthermore it is clear that the set of cluster variables for $\curly{A}_{q}[y^{\pm 1};\sigma]$ is equal to the disjoint union of the set of cluster variables for $\curly{A}_{q}$ and $\{ y,y^{-1} \}$, and so these generate all of $\curly{A}_{q}[y^{\pm 1};\sigma]$ which is therefore a graded quantum cluster algebra.

We note that we have chosen to put $y$ in degree $1$ (and its inverse in degree $-1$), to accord with the natural $\integ$-grading on this skew-Laurent extension.  However neither the element $y$ nor its inverse interacts with the quantum cluster structure coming from $\curly{A}_{q}$ so this choice was essentially arbitrary.
\end{proof}

Here we did not alter the original quantum cluster variables but in our second construction, we re-scale these by powers of an extending variable $z$.  In its most general form, this re-scaling will involve two integer column vectors as (families of) parameters, $\underline{t}=(t_{1},\dotsc ,t_{r})$ and $\underline{u}=(u_{1},\dotsc ,u_{r})$.  Given two such vectors we will denote by $\underline{t} \wedge \underline{u}$ the skew-symmetric matrix $\transpose{\underline{t}}\underline{u}-\transpose{\underline{u}}\underline{t}$, i.e. the $(i,j)$-entry of $\underline{t}\wedge \underline{u}$ is $t_{i}u_{j}-t_{j}u_{i}$.

\begin{proposition} Let $\curly{A}_{q}=\curly{A}_{q}(\underline{x}=(X_{1},\ldots,X_{r}),B,L,G)$ be a graded quantum cluster algebra such that $L$ is $d$-compatible with $B$ with $d>0$, i.e.\ we have $(\transpose{B}L)_{ij}=d\delta_{ij}$ for all mutable indices $i$ and any index $j$.  (Such a $d$ certainly exists by the definition of compatibility between $B$ and $L$; we are simply naming it explicitly.)

Now let $\tau\colon \curly{A}_{q} \to \curly{A}_{q}$ be an algebra automorphism such that for each $1\leq i \leq r$, there exists $t_{i}\in \integ$ such that $\tau(X_{i})=q^{t_{i}}X_{i}$; that is, $\tau$ acts by multiplication by powers of $q$ on the elements of the initial cluster.  We set $\underline{t}=\transpose{(t_{1},\dotsc ,t_{r})}$.  Also, let $\underline{u}=\transpose{(u_{1},\dotsc ,u_{r})}\in \integ^{r}$ be such that $\underline{t} \wedge \underline{u}$ is $f$-compatible with $B$ for some $0\leq f<d$.

Then the following is a valid set of initial data for a graded quantum cluster algebra $\tilde{\curly{A}}_{q}^{\underline{t},\underline{u}}=\curly{A}_{q}(\underline{\tilde{x}},\tilde{B},\tilde{L},\tilde{G})$ where
\begin{itemize}
\item $\underline{\tilde{x}}=(X_{1}z^{u_{1}},\ldots,X_{r}z^{u_{r}})$,
\item $\tilde{B}=B$,
\item $\tilde{L}$ is determined by the quasi-commutation data in $L$ and the vectors $\underline{t}$ and $\underline{u}$, and
\item $\tilde{G}=G$,
\end{itemize}
with $\tilde{\curly{A}}_{q}^{\underline{t},\underline{u}}$ a subalgebra of the skew-field of fractions of the skew-Laurent extension $\curly{A}_{q}[z^{\pm 1};\tau]$ of $\curly{A}_{q}$.  
\end{proposition}

We prove this validity via two lemmas, where we first establish more explicitly the matrix $\tilde{L}$ and then prove that it is compatible with $\tilde{B}=B$.

\begin{lemma}\label{l:Lminuswedge} We have $\tilde{L}=L-\underline{t}\wedge \underline{u}$.
\end{lemma}

\begin{proof} Since $\tau(X_{i})=q^{t_{i}}X_{i}$ we have that $zX_{i}=\tau(X_{i})z=q^{t_{i}}X_{i}z$ for all $i$, by the definition of a skew-Laurent extension.  For all $1\leq i,j \leq r$ we have $X_{i}X_{j}=q^{l_{ij}}X_{j}X_{i}$ and hence
\begin{align*} \tilde{X}_{i}\tilde{X}_{j} & = (X_{i}z^{u_{i}})(X_{j}z^{u_{j}}) \\
 & = X_{i}\left( (q^{t_{j}})^{u_{i}}X_{j}z^{u_{i}} \right)z^{u_{j}} \\
 & = q^{t_{j}u_{i}}\left( q^{l_{ij}}X_{j}X_{i} \right) z^{u_{i}+u_{j}} \\
 & = q^{l_{ij}+t_{j}u_{i}}X_{j}\left( X_{i}z^{u_{j}} \right) z^{u_{i}} \\
 & = q^{l_{ij}+t_{j}u_{i}}X_{j}\left( (q^{-t_{i}})^{u_{j}}z^{u_{j}}X_{i} \right)z^{u_{i}} \\
 & = q^{l_{ij}+t_{j}u_{i}-t_{i}u_{j}}(X_{j}z^{u_{j}})(X_{i}z^{u_{i}}) \\
 & = q^{l_{ij}+t_{j}u_{i}-t_{i}u_{j}}\tilde{X}_{j}\tilde{X}_{i}.
\end{align*}
Therefore $\tilde{l}_{ij}=l_{ij}+t_{j}u_{i}-t_{i}u_{j}=l_{ij}-(\underline{t}\wedge \underline{u})_{ij}$.  Note that we need $q$ not a root of unity at this point.
\end{proof}

\begin{lemma}\label{l:recaled-compatibility} The matrices $\tilde{B}$ and $\tilde{L}$ are compatible.
\end{lemma}

\begin{proof} Let $i$ be a mutable index and $j$ any index.  Then
\begin{align*} \left(\transpose{(\tilde{B}_{\text{mut}})}\tilde{L}\right)_{ij} & = \left(\transpose{(B_{\text{mut}})}(L-(\underline{t}\wedge \underline{u}))\right)_{ij} \\
 & = \left(\transpose{(B_{\text{mut}})}L\right)_{ij}-\left(\transpose{(B_{\text{mut}})}(\underline{t}\wedge \underline{u})\right)_{ij} \\
 & = d\delta_{ij}-f\delta_{ij} \\
 & = (d-f)\delta_{ij}
\end{align*}
since $L$ is $d$-compatible with $B$ and, by assumption, $\underline{t}\wedge \underline{u}$ is $f$-compatible with $B$.  Furthermore, $0\leq f<d$ so we see that $\tilde{L}$ is $(d-f)$-compatible with $B$, with $d-f$ a positive integer as required.
\end{proof}

\noindent We note that the effect of this construction is to leave the initial exchange relations unchanged but to alter the quasi-commutation relations.  That is, this construction can be thought of as a form of twisting of a quantum cluster algebra.

If $\underline{t}\wedge \underline{u}$ is $0$-compatible with $B$, the precise value of $d$ is irrelevant and we always obtain compatibility.  We note some special cases.

\begin{corollary} Let $\curly{A}_{q}$, $\underline{t}$ and $\underline{u}$ be as above.  Then
\begin{enumerate}[label=(\alph*)]
\item if $\underline{t}=0$ we have $\tilde{L}=L$, $f=0$ and hence $\curly{A}_{q}^{0,\underline{u}}\iso \curly{A}_{q}$;
\item if $\underline{u}=0$ we have $\tilde{\underline{x}}=\underline{x}$ and $\tilde{L}=L$ and hence $\curly{A}_{q}^{\underline{t},0}=\curly{A}_{q}$; and
\item if $\underline{t}$ and $\underline{u}$ are \textup{(}$\integ$\textup{-)}linearly dependent we have $\tilde{L}=L$, $f=0$ and hence $\curly{A}_{q}^{\underline{t},\underline{u}}\iso \curly{A}_{q}$.
\end{enumerate}
\end{corollary}

\begin{proof} {\ }
\begin{enumerate}[label=(\alph*)]
\item If $\underline{t}=0$ then $\underline{t}\wedge \underline{u}=0$, so that $\underline{t}\wedge \underline{u}$ is $0$-compatible with $B$.  Then the remaining claims follow.  Note that in this case the skew-Laurent extension induced by $\tau$ is a central (Laurent) extension.
\item  If $\underline{u}=0$ then again $\underline{t}\wedge \underline{u}$ is $0$-compatible with $B$ and we see immediately that $\tilde{x}=x$, $\tilde{L}=L$ and hence $\curly{A}_{q}^{\underline{t},0}=\curly{A}_{q}$.  That is, when $\underline{u}=0$ the skew-Laurent extension has no interaction with the subalgebra $\curly{A}_{q}$ (therefore the choice of $\underline{t}$ is irrelevant).
\item If $\underline{u}=\lambda\underline{t}$ for some $\lambda\in \integ$ then as above, $\underline{t}\wedge \underline{u}=0$. \qedhere
\end{enumerate}
\end{proof}

We observe that if $\underline{t}$ and $\underline{u}$ are gradings for $B$, so that $B_{i}\underline{t}=B_{i}\underline{u}=0$ for all mutable indices $i$, then 
\[ B_{i}(\underline{t}\wedge \underline{u})=\transpose{(B_{\text{mut}})}_{i}(\underline{t}\,\transpose{\underline{u}}-\underline{u}\,\transpose{\underline{t}})=0 \]
so that $\underline{t}\wedge \underline{u}$ is $0$-compatible with $B$.

Under certain conditions, the above re-scaling in fact produces a quantum cluster algebra that is a subalgebra of the skew-Laurent extension itself, and not just of the skew-field of fractions of this.  The next result gives just such a set of conditions.  However in it we will need to introduce additional powers of $q$ to our re-scaling.  The previous lemmas, as stated, still hold in this slightly more general setting: the proof of Lemma~\ref{l:Lminuswedge} requires a minor adjustment and Lemma~\ref{l:recaled-compatibility} then goes through verbatim.

\begin{theorem}\label{t:rescaledQCA} Let $\curly{A}_{q}=\curly{A}_{q}(\underline{x}=(X_{1},\ldots,X_{r}),B,L,G)$ be a graded quantum cluster algebra.  Also let $\underline{t}=\transpose{(t_{1},\dotsc ,t_{r})}\in \integ^{r}$ and $\underline{u}=\transpose{(u_{1},\dotsc ,u_{r})}\in \integ^{r}$ be such that $\underline{t}$ and $\underline{u}$ are gradings for $B$, i.e.\ $B_{i}\underline{t}=B_{i}\underline{u}=0$ for every mutable index $i$.

Then the following initial data determine a graded quantum cluster algebra $\tilde{\curly{A}}_{q}^{\underline{t},\underline{u}}=\curly{A}_{q}(\underline{\tilde{x}},\tilde{B},\tilde{L},\tilde{G})$ where
\begin{itemize}
\item $\underline{\tilde{x}}=(q^{\frac{t_{1}u_{1}}{2}}X_{1}z^{u_{1}},\ldots,q^{\frac{t_{r}u_{r}}{2}}X_{r}z^{u_{r}})$,
\item $\tilde{B}=B$,
\item $\tilde{L}=L-\underline{t}\wedge \underline{u}$, and
\item $\tilde{G}=G+\underline{u}$,
\end{itemize}
with $\tilde{\curly{A}}_{q}^{\underline{t},\underline{u}}$ a subalgebra of the skew-Laurent extension $\curly{A}_{q}[z^{\pm 1};\tau]$ of $\curly{A}_{q}$ whose automorphism $\tau$ is induced by $\underline{t}$, i.e.\ $\tau\colon \curly{A}_{q} \to \curly{A}_{q}$ is the algebra automorphism such that for each $1\leq i \leq r$, $\tau(X_{i})=q^{t_{i}}X_{i}$.
\end{theorem}

\begin{proof} Our strategy for this proof will be to consider seeds augmented by the extra data assumed in the theorem and to show that these extended seeds behave appropriately under mutation.  More precisely, consider as initial data the tuple
\[ (\underline{\tilde{x}}=(\tilde{X}_{1},\dotsc ,\tilde{X}_{r}),\tilde{B},\tilde{L},\tilde{G},\underline{t},\underline{u}),\]
where each component is as defined in, and satisfies the conditions of, the statement of the theorem and in particular
\[ \tilde{X}_{i}=q^{\frac{t_{i}u_{i}}{2}}X_{i}z^{u_{i}}. \]
We will call such a tuple a \emph{re-scaled} seed and we first establish that this data is valid for defining a graded quantum cluster algebra.

We observed above that $\underline{t}$ and $\underline{u}$ being gradings implies that $\underline{t}\wedge \underline{u}$ is $0$-compatible and so we apply the lemmas to see that $\tilde{B}$ and $\tilde{L}$ are compatible.  Since $\tilde{B}=B$ and $\tilde{G}=G+\underline{u}$ with both $G$ and $\underline{u}$ being gradings for $B$, the grading condition follows.  That is, the first four components of the re-scaled initial seed are indeed valid data for the construction of a graded quantum cluster algebra.  We note that the choice of $\tilde{G}=G+\underline{u}$ is natural, setting the degree of the re-scaled variable $q^{\frac{t_{i}u_{i}}{2}}X_{i}z^{u_{i}}$ to be the sum of the degree of $X_{i}$ (as described by $G$) and the power of $z$, namely $u_{i}$.

Mutation of re-scaled seeds is defined in the obvious way: the cluster $\tilde{\underline{x}}$ is mutated via the exchange relations determined by $\tilde{B}$ and $\tilde{L}$ as usual, $\tilde{B}$ and $\tilde{L}$ are mutated in the same way as for ungraded quantum cluster algebras via the corresponding matrix $E$ and the three gradings $\tilde{G}$, $\underline{t}$ and $\underline{u}$ are mutated as described in Section~\ref{s:gradedQCAs}, namely by multiplication by $\transpose{E}$.  Then Lemma~\ref{l:mutofgrading} assures us that the mutations of $\underline{t}$ and $\underline{u}$ are gradings for the corresponding mutation of $\tilde{B}$, so we see that the mutation of a re-scaled seed again satisfies the compatibilities and assumptions of the theorem, except that we need to see that the form of the mutated cluster variables is the same as that described above.  That is, we wish to show that $\tilde{X}_{i}^{\prime}=q^{\frac{t_{i}^{\prime}u_{i}^{\prime}}{2}}(X_{i}^{\prime})z^{u_{i}^{\prime}}$.

In order to verify this we argue similarly to the proof of Proposition~\ref{p:rescaleQCAbypowersofq} and first consider the exchange monomials arising from a re-scaled seed, i.e.\ 
\[ \tilde{M}(a_{1},\dotsc ,a_{r}) \defeq q^{\frac{1}{2}\sum_{i<j} a_{i}a_{j}\tilde{l}_{ji}}\tilde{X}_{1}^{a_{1}}\dotsm \tilde{X}_{r}^{a_{r}}. \]  
We have $\tilde{X}_{i}=q^{\frac{t_{i}u_{i}}{2}}X_{i}z^{u_{i}}$ and it is straightforward to verify from the quasi-commutation relations between $z$ and the $X_{i}$ and the equation $\tilde{L}=L-\underline{t}\wedge \underline{u}$ that we have
\[ \tilde{M}(a_{1},\dotsc ,a_{r})=q^{\frac{1}{2}\sum_{i=1}^{r}\sum_{j=1}^{r} a_{i}a_{j}t_{i}u_{j}}M(a_{1},\dotsc ,a_{r})z^{\sum_{i=1}^{r} u_{i}a_{i}}.\]
Details may be found in the Appendix on page~\pageref{appendix}.  Writing $\underline{a}=(a_{1},\dotsc ,a_{r})$, we may reformulate this as
\[ \tilde{M}(\underline{a})=q^{\frac{(\underline{a}\cdot \underline{t})(\underline{a} \cdot \underline{u})}{2}}M(\underline{a})z^{\underline{a}\cdot \underline{u}}.\]

Then for $k$ a mutable index, recalling that $\tilde{B}=B$ we see that mutation in the direction $k$ from the re-scaled seed yields
\begin{align*}
\tilde{X}_{k}^{\prime} & =\tilde{M}(\underline{b}_{k}^{+})+\tilde{M}(\underline{b}_{k}^{-}) \\
 & = q^{\frac{(\underline{b}_{k}^{+}\cdot \underline{t})(\underline{b}_{k}^{+} \cdot \underline{u})}{2}}M(\underline{b}_{k}^{+})z^{\underline{b}_{k}^{+}\cdot \underline{u}}+q^{\frac{(\underline{b}_{k}^{-}\cdot \underline{t})(\underline{b}_{k}^{-} \cdot \underline{u})}{2}}M(\underline{b}_{k}^{-})z^{\underline{b}_{k}^{-}\cdot \underline{u}} \\
 & = q^{\frac{{t}_{k}^{\prime}{u}_{k}^{\prime}}{2}}M(\underline{b}_{k}^{+})z^{{u}_{k}^{\prime}}+q^{\frac{{t}_{k}^{\prime}{u}_{k}^{\prime}}{2}}M(\underline{b}_{k}^{-})z^{{u}_{k}^{\prime}} \\
 & = q^{\frac{{t}_{k}^{\prime}{u}_{k}^{\prime}}{2}}\left(M(\underline{b}_{k}^{+})+M(\underline{b}_{k}^{-})\right)z^{{u}_{k}^{\prime}} \\
 & = q^{\frac{{t}_{k}^{\prime}{u}_{k}^{\prime}}{2}}(X_{k}^{\prime})z^{{u}_{k}^{\prime}} 
\end{align*}
as desired.  Here we have again used the fact that $\underline{b}_{k}^{+}\cdot \underline{v}=\underline{b}_{k}^{-}\cdot \underline{v}$ for any grading $\underline{v}$ for $B$ and the equality of both of these with the $k$th entry of the mutation of $\underline{v}$ (as noted after Lemma~\ref{l:mutofgradinggivenbyE}), as well as the fact that the corresponding exchange relation in the original quantum cluster algebra $\curly{A}_{q}$ is $X_{k}^{\prime}=M(\underline{b}_{k}^{+})+M(\underline{b}_{k}^{-})$.  Note that the power of $z$ occurring is exactly the degree of $\tilde{X}_{k}^{\prime}$ (or equivalently of $X_{k}^{\prime}$) for the grading induced by $\underline{u}$.

That is, mutation of a re-scaled seed produces another re-scaled seed.  Therefore iterated mutation from the re-scaled seed of the statement produces a graded quantum cluster algebra all of whose quantum cluster variables are contained in the skew-Laurent extension $\curly{A}_{q}[z^{\pm 1};\tau]$, i.e.\ no localisation of the latter is required.
\end{proof}

From the proof of this theorem, we see the following.

\begin{corollary}\label{c:formofrescaledQCV} With notation as in the preceding Theorem, there is a bijection $\phi$ between the sets of quantum cluster variables for the quantum cluster algebras $\curly{A}_{q}$ and $\tilde{\curly{A}}_{q}^{\underline{t},\underline{u}}$.  Furthermore, under this bijection, every quantum cluster variable $\tilde{X}$ of the quantum cluster algebra $\tilde{\curly{A}}_{q}^{\underline{t},\underline{u}}$ is of the form $q^{a}Xz^{b}$ with $a,b\in \integ$ and $X=\phi^{-1}(\tilde{X})$ the corresponding quantum cluster variable in $\curly{A}_{q}$. \qed
\end{corollary}

\begin{remark} The quantum cluster algebra structure $\tilde{\curly{A}}_{q}^{\underline{t},\underline{u}}$ from this theorem is a quantum cluster algebra structure on a proper subalgebra of the skew-Laurent extension $\curly{A}_{q}[z^{\pm 1};\tau]$; in general this subalgebra is not $\curly{A}_{q}$.  However we could easily extend this to a quantum cluster algebra structure on the whole skew-Laurent extension by adding $z$ and $z^{-1}$ as coefficients in the same manner as in Proposition~\ref{p:skewLaurent-extra-coeffs}, as the only issue here is the absence of the generators $z$ and $z^{-1}$.
\end{remark}

We note that if $A$ is a $\mathbb{K}$-algebra and $\tau$ an automorphism of $A$ then there is an algebra isomorphism of $\left((A[z_{1}^{\pm 1};\tau])[z_{2}^{\pm 1};\tau]\right)/(z_{1}-z_{2})$ with $A[z_{1}^{\pm 1};\tau]$.  That is, if we make a two-fold skew-Laurent extension of $A$ using first the automorphism $\tau$ and then using $\tau$ extended to $A[z_{1}^{\pm 1};\tau]$ as the identity on $z_{1}^{\pm 1}$, then taking the quotient that identifies the two variables yields an algebra isomorphic to a single extension using $\tau$.

The next lemma observes that one may reverse the above scaling procedure.

\begin{lemma} Let $\curly{A}_{q}=\curly{A}_{q}(\underline{x},B,L,G)$ be a graded quantum cluster algebra and let $\tilde{\curly{A}}_{q}^{\underline{t},\underline{u}} \subseteq \curly{A}_{q}[z^{\pm 1};\tau]$ be the quantum cluster algebra structure obtained from $\curly{A}_{q}$ by the construction of the preceding theorem.  Then $(\tilde{\curly{A}}_{q}^{\underline{t},\underline{u}})^{\underline{t},-\underline{u}}=\curly{A}_{q}$, where the former is viewed as a subalgebra of $\curly{A}_{q}[z^{\pm 1};\tau]$ under the isomorphism described above.
\end{lemma}

\begin{proof} The quantum cluster algebra $(\tilde{\curly{A}}_{q}^{\underline{t},\underline{u}})^{\underline{t},-\underline{u}}$ is obtained from $\tilde{\curly{A}}_{q}^{\underline{t},\underline{u}}$ by the construction of the theorem.  Using the isomorphism of the quotient of the two-fold extension with the single extension described above we may abuse notation and write $z$ for both extending variables.  We then see that $(\tilde{\curly{A}}_{q}^{\underline{t},\underline{u}})^{\underline{t},-\underline{u}}$ has initial data
\begin{itemize}
\item $\tilde{\tilde{\underline{x}}}=(q^{\frac{-t_{1}u_{1}}{2}}(q^{\frac{t_{1}u_{1}}{2}}X_{1}z^{u_{1}})z^{-u_{1}},\dotsc ,q^{\frac{-t_{r}u_{r}}{2}}(q^{\frac{t_{r}u_{r}}{2}}X_{r}z^{u_{r}})z^{-u_{r}})=(X_{1},\dotsc ,X_{r})=\underline{x}$,
\item $\tilde{\tilde{B}}=B$,
\item $\tilde{\tilde{L}}=(L-\underline{t}\wedge \underline{u})-\underline{t}\wedge (-\underline{u})=L$ and
\item $\tilde{\tilde{G}}=(G+\underline{u})+(-\underline{u})=G$.
\end{itemize}
That is, $(\tilde{\curly{A}}_{q}^{\underline{t},\underline{u}})^{\underline{t},-\underline{u}}$ has the same initial data as $\curly{A}_{q}$ so yields the same quantum cluster algebra.
\end{proof}

This completes our general theory of graded quantum cluster algebras.  Now we turn to our application, the existence of a quantum cluster algebra structure on the quantum Grassmannians.

\section{The quantum cluster algebra structure on quantum matrices}\label{s:KqMatisQCA}

As noted above, the work of Gei\ss, Leclerc and Schr\"{o}er (\cite[Corollary~12.10]{GLS-QuantumPFV}) has given a quantum cluster algebra structure on quantum matrices $\KqMat{k}{j}$.  We use $j$ rather than $n$ here as this is the notation of \cite{GLS-QuantumPFV} and also we will want to consider $\KqGr{k}{n}$ and its relationship with $\KqMat{k}{n-k}$ subsequently; it will simplify the presentation in this section to use $j$ rather than $n-k$.

Our aim is to lift this to a quantum cluster algebra structure on the corresponding quantum Grassmannian $\KqGr{k}{k+j}$, in a similar fashion to \cite[\S 10]{GLS-PFV}, so we record here the initial data provided by the construction in \cite{GLS-QuantumPFV}.  (This section is an expansion of \cite[Example~12.11]{GLS-QuantumPFV}, which describes the case $k=j=3$.)

Let $m=k+j-1$; for comparison with \cite[\S12.4]{GLS-QuantumPFV}, our parameter $m$ is their $n$.  The construction of the quantum cluster algebra structure on $\KqMat{k}{j}$ is via the module category of the preprojective algebra $\Lambda=\Lambda(A_{m})$ associated to the Dynkin diagram $A_{m}$.  For a description of the algebra $\Lambda$ and its representation theory, including Auslander--Reiten quivers for $2\leq m \leq 4$, we refer the reader to \cite{GLS-Semicanonical}.  In what follows, we will state well-known properties of this algebra and its module category without proof.  

We first need to construct the projective modules for $\Lambda$.  A basis for the $i$th projective module $P_{i}$ is given by the set of paths leaving the vertex $i$ (modulo the preprojective relation).  The $r$th Loewy layer of $P_{i}$ consists of the simple modules corresponding to the vertices at the ends of paths of length $r-1$ and so we see that $P_{i}$ has simple socle $i$ and simple top $m-i+1$.  Here we use the common notation of having the vertex labels denote the simple modules for path algebras and quotients of these (see for example \cite{ASS-Book}); we will also use $S_{i}$ for this when this is clearer.  

\vfill
\pagebreak
In general $P_{i}$ has the form of an $i$ by $m-i+1$ rectangle:
\begin{center}
\scriptsize
\scalebox{0.8}{\begin{tikzpicture}[node distance=1cm,on grid]

\node (11) at (0,0) {$m-i+1$}; 
\node (21) [below left=of 11] {$m-i$};
\node (22) [below right=of 11] {$m-i+2$};
\node (31) [below left=of 21] {$m-i-1$};
\node (32) [below left=of 22] {$m-i+1$};
\node (33) [below right=of 22] {$m-i+3$};
\node (41) [below left=of 31] {$\adots$};
\node (42) [below left=of 32] {$\adots$};
\node (43) [below left=of 33] {$\ddots$};
\node (44) [below right=of 33] {$\ddots$};
\node (51) [below left=of 41] {$\adots$};
\node (52) [below left=of 42] {$\adots$};
\node (53) [below left=of 43] {};
\node (54) [below left=of 44] {$\ddots$};
\node (55) [below right=of 44] {$m-1$};
\node (61) [below left=of 51] {$2$};
\node (62) [below left=of 52] {$\adots$};
\node (63) [below left=of 53] {};
\node (64) [below left=of 54] {};
\node (65) [below left=of 55] {$m-2$};
\node (66) [below right=of 55] {$m$};
\node (71) [below left=of 61] {$1$};
\node (72) [below left=of 62] {$3$};
\node (73) [below left=of 63] {};
\node (74) [below left=of 64] {};
\node (75) [below left=of 65] {$\adots$};
\node (76) [below right=of 65] {$m-1$};
\node (81) [below right=of 71] {$2$};
\node (82) [below right=of 72] {$\ddots$};
\node (83) [below right=of 73] {};
\node (84) [below right=of 74] {$\adots$};
\node (85) [below right=of 75] {$\adots$};
\node (91) [below right=of 81] {$\ddots$};
\node (92) [below right=of 82] {$\ddots$};
\node (93) [below right=of 83] {$\adots$};
\node (94) [below right=of 84] {$\adots$};
\node (101) [below right=of 91] {$i-2$};
\node (102) [below right=of 92] {$i$};
\node (103) [below right=of 93] {$i+2$};
\node (111) [below right=of 101] {$i-1$};
\node (112) [below right=of 102] {$i+1$};
\node (121) [below right=of 111] {$i$};

\end{tikzpicture}}
\normalsize
\end{center}
Note that if $k=j$ then $m=2k-1$ is odd and $P_{k}$ is self-dual---the rectangular shape depicted above is then a square.  Hence for $k=j=3$, so $m=5$, the projective modules are

\begin{center}
\scalebox{1}{\begin{tikzpicture}[node distance=0.5cm,on grid]

\node (1-11) at (0,0) {$5$}; 
\node (1-21) [below left=of 1-11] {$4$};
\node (1-31) [below left=of 1-21] {$3$};
\node (1-41) [below left=of 1-31] {$2$};
\node (1-51) [below left=of 1-41] {$1$};
\node (a) [above left=of 1-21] {};
\node (1-label) [above=of a,yshift=0.5cm] {$P_{1}$};

\node (2-11) at (3,0) {$4$}; 
\node (2-21) [below left=of 2-11] {$3$};
\node (2-22) [below right=of 2-11] {$5$};
\node (2-31) [below left=of 2-21] {$2$};
\node (2-32) [below left=of 2-22] {$4$};
\node (2-41) [below left=of 2-31] {$1$};
\node (2-42) [below left=of 2-32] {$3$};
\node (2-52) [below left=of 2-42] {$2$};
\node (b) [above left=of 2-21] {};
\node (2-label) [above=of b,yshift=0.5cm] {$P_{2}$};

\node (3-11) at (6,0) {$3$}; 
\node (3-21) [below left=of 3-11] {$2$};
\node (3-22) [below right=of 3-11] {$4$};
\node (3-31) [below left=of 3-21] {$1$};
\node (3-32) [below left=of 3-22] {$3$};
\node (3-33) [below right=of 3-22] {$5$};
\node (3-42) [below left=of 3-32] {$2$};
\node (3-43) [below left=of 3-33] {$4$};
\node (3-53) [below left=of 3-43] {$3$};
\node (3-label) [above=of 3-11,yshift=0.5cm] {$P_{3}$};

\node (4-11) at (9,0) {$2$}; 
\node (4-21) [below left=of 4-11] {$1$};
\node (4-22) [below right=of 4-11] {$3$};
\node (4-32) [below left=of 4-22] {$2$};
\node (4-33) [below right=of 4-22] {$4$};
\node (4-43) [below left=of 4-33] {$3$};
\node (4-44) [below right=of 4-33] {$5$};
\node (4-54) [below left=of 4-44] {$4$};
\node (d) [above right=of 4-22] {};
\node (4-label) [above=of d,yshift=0.5cm] {$P_{4}$};

\node (5-11) at (12,0) {$1$}; 
\node (5-21) [below right=of 5-11] {$2$};
\node (5-31) [below right=of 5-21] {$3$};
\node (5-41) [below right=of 5-31] {$4$};
\node (5-51) [below right=of 5-41] {$5$};
\node (e) [above right=of 5-21] {};
\node (5-label) [above=of e,yshift=0.5cm] {$P_{5}$};

\end{tikzpicture}}
\end{center}

Now it is well-known that $\KqMat{k}{j}$ is isomorphic to the algebra $U_{q}(\mathfrak{n}(w))$ (also commonly denoted $U_{q}^{+}[w]$) associated to $\mathfrak{g}=\mathfrak{sl}_{m+1}$ where $w$ is the Weyl group word with reduced decomposition \[ w=(s_{j}s_{j-1}\cdots s_{1})(s_{j+1}s_{j}\cdots s_{2})\cdots (s_{m}s_{m-1}\cdots s_{k}). \]  This may be found in \cite{MeriauxCauchon}, for example.  Let \[ \underline{i}=(k,k+1,\ldots,m,k-1,k,\ldots,m-1,\ldots,1,2,\ldots,j) \] be the sequence of indices for the above reduced decomposition for $w$; note that we have chosen the reverse order to that in \cite[\S12.4]{GLS-QuantumPFV}.  It is convenient to render this as a $k\cross j$ matrix $(i_{(\alpha,\beta)})$ with $i_{(\alpha,\beta)}=k-\alpha+\beta$, for $1\leq \alpha \leq k$, $1 \leq \beta \leq j$, \ie
\[ \underline{i} = \begin{pmatrix} k & k+1 & \cdots & m-1 & m \\ k-1 & k & \cdots & m-2 & m-1 \\ \vdots & \vdots & \ddots & \vdots & \vdots \\ 2 & 3 & \cdots & j & j+1 \\ 1 & 2 & \cdots & j-1 & j \end{pmatrix}. \]

There is a natural total order on the set of indices of the matrix $\underline{i}$, given by $(\alpha,\beta)<(\alpha',\beta')$ if and only if $\alpha<\alpha'$ or $(\alpha=\alpha'$ and $\beta<\beta')$; that is, the ordering is lexicographical in each coordinate, taking the first coordinate first.  We extend this to a total order on the set $\{ (\alpha,\beta) \mid 1\leq \alpha \leq k,\ 1\leq \beta \leq j\} \union \{0,kj+1\}$ by $0<(\alpha,\beta)<kj+1$ for all pairs $(\alpha,\beta)$.  For an index $(\alpha,\beta)$, we define 
\begin{eqnarray*}
(\alpha,\beta)^{-} & = & \max \left( \{ 0 \} \union \left\{ (\gamma,\delta)<(\alpha,\beta) \mid i_{(\gamma,\delta)}=i_{(\alpha,\beta)} \right\}\right) \\ & = & \begin{cases} 0 & \text{if}\ \alpha=1\ \text{or}\ \beta=1 \\ (\alpha-1,\beta-1) & \text{otherwise} \end{cases} \\
(\alpha,\beta)^{+} & = & \min\left(\{kj+1\} \union \left\{(\gamma,\delta)>(\alpha,\beta) \mid i_{(\gamma,\delta)}=i_{(\alpha,\beta)} \right\}\right) \\ & = & \begin{cases} kj+1 & \text{if}\ \alpha=k\ \text{or}\ \beta=j \\ (\alpha+1,\beta+1) & \text{otherwise} \end{cases}
\end{eqnarray*}
The frozen indices (that is, those indices that correspond to coefficients in the quantum cluster algebra structure) are exactly the $(\alpha,\beta)$ with $(\alpha,\beta)^{+}=kj+1$, \ie when $\alpha=k$ or $\beta=j$.

The initial seed is constructed from the module category as follows.  The subcategory of $\mathrm{mod}(\Lambda)$ corresponding to the word $w$ above, which we denote by $\curly{C}_{w}$, is the subcategory generated by the projective module $P_{k}$.  Certain quotients of $P_{k}$ give the modules corresponding to the standard generators $X_{ab}$ of $\KqMat{k}{j}$ and an iterated socle construction is used to produce modules in this subcategory that correspond to elements of the initial seed.  More precisely, for each pair $(a,b)$ with $1\leq a \leq k$ and $1\leq b \leq j$, the module $P_{k}$ has a unique quotient $M_{(a,b)}$ whose dimension vector is $\underline{\boldsymbol{e}}_{a}+\underline{\boldsymbol{e}}_{a+1}+\cdots +\underline{\boldsymbol{e}}_{m-b+1}$ and this quotient corresponds to $X_{ab}$.  From the above description of $P_{k}$ we see that these modules correspond to segments of the top edges of the rectangle describing $P_{k}$ that include the top (which is isomorphic to $S_{m-k+1}=S_{j}$).  For $k=j=3$ ($m=5$) we have

\begin{center}
\scalebox{1}{\begin{tikzpicture}[node distance=0.5cm,on grid]

\node (11-3) at (0,0) {$3$}; 
\node (11-label) at (-2,0) {$M_{(3,3)}=$};

\node (12-3) at (4,0) {$3$}; 
\node (12-4) [below right=of 12-3] {$4$};
\node (12-label) at (4-2,0) {$M_{(3,2)}=$};

\node (13-3) at (8,0) {$3$}; 
\node (13-4) [below right=of 13-3] {$4$};
\node (13-5) [below right=of 13-4] {$5$};
\node (13-label) at (8-2,0) {$M_{(3,1)}=$};

\node (21-3) at (0,-2) {$3$}; 
\node (21-2) [below left=of 21-3] {$2$};
\node (21-label) at (0-2,0-2) {$M_{(2,3)}=$};

\node (22-3) at (4,-2) {$3$}; 
\node (22-2) [below left=of 22-3] {$2$};
\node (22-4) [below right=of 22-3] {$4$};
\node (22-label) at (4-2,0-2) {$M_{(2,2)}=$};

\node (23-3) at (8,-2) {$3$}; 
\node (23-2) [below left=of 23-3] {$2$};
\node (23-4) [below right=of 23-3] {$4$};
\node (23-5) [below right=of 23-4] {$5$};
\node (23-label) at (8-2,0-2) {$M_{(2,1)}=$};

\node (31-3) at (0,-4) {$3$}; 
\node (31-2) [below left=of 31-3] {$2$};
\node (31-1) [below left=of 31-2] {$1$};
\node (31-label) at (0-2,0-4) {$M_{(1,3)}=$};

\node (32-3) at (4,-4) {$3$}; 
\node (32-2) [below left=of 32-3] {$2$};
\node (32-4) [below right=of 32-3] {$4$};
\node (32-1) [below left=of 32-2] {$1$};
\node (32-label) at (4-2,0-4) {$M_{(1,2)}=$};

\node (33-3) at (8,-4) {$3$}; 
\node (33-2) [below left=of 33-3] {$2$};
\node (33-4) [below right=of 33-3] {$4$};
\node (33-1) [below left=of 33-2] {$1$};
\node (33-5) [below right=of 33-4] {$5$};
\node (33-label) at (8-2,0-4) {$M_{(1,1)}=$};

\end{tikzpicture}}
\end{center}

To construct the modules corresponding to the initial seed, we need the following construction.  Given a module $W$, we define
\begin{itemize}
\item $\mathrm{soc}_{(l)}(W)\defeq {\displaystyle \sum_{\substack{U\leq W \\ U\iso S_{l}}} U}$ and
\item $\mathrm{soc}_{(l_{1},l_{2},\ldots,l_{s})}(W)\defeq W_{s}$ where the chain of submodules $0\subseteq W_{1} \subseteq W_{2} \subseteq \cdots \subseteq W_{s} \subseteq W$ is such that $W_{p}/W_{p-1} \iso \mathrm{soc}_{(l_{p})}(W/W_{p-1})$.
\end{itemize}
Then for $1\leq s \leq l(w)=kj$, we define $V_{s} \defeq \mathrm{soc}_{(i_{s},i_{s-1},\ldots,i_{1})}(P_{i_{s}})$.  Thus for $k=j=3$ and our choice of reduced expression $\underline{i}$ above, $V_{1}=\mathrm{soc}_{(3)}(P_{3})=\mathrm{soc}(P_{3})=S_{3}$.

Similarly $V_{2}=\mathrm{soc}_{(4,3)}(P_{4})$ is defined by the chain $0\subseteq W_{1} \subseteq W_{2}=V_{2} \subseteq P_{4}$ with $W_{1}=\mathrm{soc}_{(4)}(P_{4})=S_{4}$ and $W_{2}/W_{1}=\mathrm{soc}_{(3)}(P_{4}/W_{1})=S_{3}$; that is, $V_{2}$ has two layers, a simple top and a simple socle isomorphic to $S_{3}$ and $S_{4}$ respectively.  Arranging the modules $V_{s}$ in the same way as we did for the indices $i_{s}$, it is natural to re-number these as $V_{(\alpha,\beta)}$ for $1\leq \alpha \leq k$, $1\leq \beta \leq j$, and we see that the modules corresponding to the initial seed for this case are as follows:

\begin{center}
\scalebox{1}{\begin{tikzpicture}[node distance=0.5cm,on grid]

\node (11-3) at (0,0) {$3$}; 
\node (11-label) at (-2,0) {$V_{(1,1)}=$};

\node (12-3) at (4,0) {$3$}; 
\node (12-4) [below right=of 12-3] {$4$};
\node (12-label) at (4-2,0) {$V_{(1,2)}=$};

\node (13-3) at (8,0) {$3$}; 
\node (13-4) [below right=of 13-3] {$4$};
\node (13-5) [below right=of 13-4] {$5$};
\node (13-label) at (8-2,0) {$V_{(1,3)}=$};

\node (21-3) at (0,-2) {$3$}; 
\node (21-2) [below left=of 21-3] {$2$};
\node (21-label) at (0-2,0-2) {$V_{(2,1)}=$};

\node (22-3) at (4,-2) {$3$}; 
\node (22-2) [below left=of 22-3] {$2$};
\node (22-4) [below right=of 22-3] {$4$};
\node (22-33) [below right=of 22-2] {$3$};
\node (22-label) at (4-2,0-2) {$V_{(2,2)}=$};

\node (23-3) at (8,-2) {$3$}; 
\node (23-2) [below left=of 23-3] {$2$};
\node (23-4) [below right=of 23-3] {$4$};
\node (23-5) [below right=of 23-4] {$5$};
\node (23-33) [below right=of 23-2] {$3$};
\node (23-44) [below right=of 23-33] {$4$};
\node (23-label) at (8-2,0-2) {$V_{(2,3)}=$};

\node (31-3) at (0,-4) {$3$}; 
\node (31-2) [below left=of 31-3] {$2$};
\node (31-1) [below left=of 31-2] {$1$};
\node (31-label) at (0-2,0-4) {$V_{(3,1)}=$};

\node (32-3) at (4,-4) {$3$}; 
\node (32-2) [below left=of 32-3] {$2$};
\node (32-4) [below right=of 32-3] {$4$};
\node (32-1) [below left=of 32-2] {$1$};
\node (32-33) [below right=of 32-2] {$3$};
\node (32-22) [below right=of 32-1] {$2$};
\node (32-label) at (4-2,0-4) {$V_{(3,2)}=$};

\node (33-3) at (8,-4) {$3$}; 
\node (33-2) [below left=of 33-3] {$2$};
\node (33-4) [below right=of 33-3] {$4$};
\node (33-1) [below left=of 33-2] {$1$};
\node (33-5) [below right=of 33-4] {$5$};
\node (33-22) [below right=of 33-1] {$2$};
\node (33-33) [below right=of 33-2] {$3$};
\node (33-44) [below right=of 33-33] {$4$};
\node (33-333) [below right=of 33-22] {$3$};
\node (33-label) at (8-2,0-4) {$V_{(3,3)}=$};

\end{tikzpicture}}
\end{center}

To obtain the element of $\KqMat{k}{j}$ corresponding to the modules $V_{(\alpha,\beta)}$, we note that the construction of the $V_{(\alpha,\beta)}$ is such that $V_{(\alpha,\beta)}/V_{(\alpha,\beta)^{-}}=V_{(\alpha,\beta)}/V_{(\alpha-1,\beta-1)} \iso M_{(k-\alpha+1,n-\beta+1)}$ (the natural indexings of the $V_{(\alpha,\beta)}$ and the $M_{(a,b)}$ are opposed to each other, unfortunately).  Then returning to our running example with $k=j=3$, $V_{(1,1)}/0=M_{(3,3)}$ so $V_{(1,1)}$ corresponds to $X_{33}$.  A module $V_{(\alpha,\beta)}$ need not correspond to a generator: $V_{(2,2)}$ is an extension of $M_{(3,3)}$ by $M_{(2,2)}$ and corresponds to the quantum minor $\minor{23}{23}$.  Similarly $V_{(3,3)}=P_{3}$ is an extension of $V_{(2,2)}$ by $M_{(1,1)}$ and corresponds to the quantum minor $\minor{123}{123}$.  

We may describe the initial cluster coming from this construction, which we will call $\curly{M}(k,j)$, explicitly as follows.

\begin{definition}\label{d:qGLScluster}
For $1\leq r\leq k$ and $1\leq s \leq j$, define the sets
\begin{align*}
R(r,s) & = \{ k-r+1,k-r+2,\ldots ,k-r+s \} \intersection \{1,\ldots,k \} \\
C(r,s) & = \{ j-s+1,j-s+2,\ldots ,j-s+r \} \intersection \{1,\ldots,j \} 
\end{align*}
Then we define $\curly{M}(k,j)=\{ \minor{R(r,s)}{C(r,s)} \mid 1\leq r\leq k, 1\leq s\leq j\}$.  It is natural to give $\curly{M}(k,j)$ as a $k\cross j$ array (as we have for $\underline{i}$), where its $(r,s)$-entry, which we denote $\curly{M}_{kj}(r,s)$, is the quantum minor with row set $R(r,s)$ and column set $C(r,s)$.  Should we need to consider $\curly{M}(k,j)$ as a sequence, its $\left((r-1)j+s\right)$-entry is $\minor{R(r,s)}{C(r,s)}$.
\end{definition} 

\begin{remark}  The above association of modules to minors follows from well-known isomorphisms, such as the isomorphism of $\KqMat{k}{j}$ with $U_{q}(\mathfrak{n}(w))$ for the above $w$ as in \cite{MeriauxCauchon}, and those in the paper \cite{GLS-QuantumPFV}.  We note that for $\mathbb{K}_{q}[SL_{m+1}]$, the generalized quantum minors of \cite{BZ-QCA} are the usual quantum minors (analogous to the fact that the generalized minors of Fomin and Zelevinsky \cite{FZ-DoubleBruhat} coincide with the usual ones for $SL_{m+1}$).  Then the unipotent quantum minors in the paper of Gei\ss, Leclerc and Schr\"{o}er (\cite[\S 5]{GLS-QuantumPFV}) are generalized quantum minors divided by certain principal quantum minors.  Following through the correspondence of these with dual PBW basis elements (in $U_{q}(\mathfrak{n}(w))$) and thence through the isomorphism of M\'{e}riaux and Cauchon (\cite{MeriauxCauchon}), we do indeed obtain the (usual) quantum minors in $\KqMat{k}{j}$.
\end{remark}

The arrows in the exchange quiver for the initial seed are given by the combinatorial data associated to the reduced expression $\underline{i}$.  Following through the definitions in \cite[\S9.4]{GLS-QuantumPFV} in this case yields the following description of these:
\begin{itemize}
\item $(\alpha,\beta) \to (\alpha,\beta+1)$,
\item $(\alpha,\beta) \to (\alpha+1,\beta)$ and
\item $(\alpha,\beta) \to (\alpha-1,\beta-1)$,
\end{itemize}
where an arrow is defined only if both its start and end points are (thus there is no arrow $(1,1)\to (0,0)$, for example) and any arrows between indices for coefficients are suppressed.  We note that these are exactly opposed to the natural inclusion and projection homomorphisms on the corresponding modules.  

The quasi-commutation data is also encoded categorically: indexing by pairs as above, the matrix $L$ has entries
\[ l_{(\alpha,\beta),(\gamma,\delta)}=\dim \mathrm{Hom}_{\Lambda}(V_{(\alpha,\beta)},V_{(\gamma,\delta)})-\dim \mathrm{Hom}_{\Lambda}(V_{(\alpha,\beta)},V_{(\gamma,\delta)}).\]  Alternatively, this data can be obtained combinatorially (\cite[Proposition~10.3]{GLS-QuantumPFV}).  The compatibility of the matrix corresponding to the arrows in the exchange quiver and the quasi-commutation matrix is shown in Proposition~10.1 of \cite{GLS-QuantumPFV}.

Putting this all together, the initial cluster variables and their exchange quiver in $\KqMat{k}{j}$ are as illustrated in Figure~\ref{fig:initialseed}.  In Figure~\ref{fig:initialseed33} we show this for our running example with $k=j=3$.

\begin{figure}[t]
\begin{center}
{\footnotesize
\scalebox{0.7}{\begin{tikzpicture}[node distance=3cm,on grid,>=angle 90]

\node (11) at (0,0) {$(kj)$}; 
\node (12) [right=of 11] {$\left(k(j-1)\right)$};
\node (13) [right=of 12] {$\left(k(j-2)\right)$};
\node (14) [right=of 13] {};
\node (15) [right=of 14] {};
\node (16) [right=of 15] {$(k3)$};
\node (17) [right=of 16] {$(k2)$};
\node (18) [right=of 17,rectangle,draw=black,thick] {$(k1)$};

\node (21) [below=of 11] {$\left((k-1)j\right)$}; 
\node (22) [right=of 21] {$\begin{bmatrix} (j-1)j \\ (k-1)k \end{bmatrix}$};
\node (23) [right=of 22] {$\begin{bmatrix} (j-2)(j-1) \\ (k-1)k \end{bmatrix}$};
\node (24) [right=of 23] {};
\node (25) [right=of 24] {};
\node (26) [right=of 25] {$\begin{bmatrix} 34 \\ (k-1)k \end{bmatrix}$};
\node (27) [right=of 26] {$\begin{bmatrix} 23 \\ (k-1)k \end{bmatrix}$};
\node (28) [right=of 27,rectangle,draw=black,thick] {$\begin{bmatrix} 12 \\ (k-1)k \end{bmatrix}$};

\node (31) [below=of 21] {$\left((k-2)j\right)$}; 
\node (32) [right=of 31] {$\begin{bmatrix} (j-1)j \\ (k-2)(k-1) \end{bmatrix}$};
\node (33) [right=of 32] {$\begin{bmatrix} (j-2)(j-1)j \\ (k-2)(k-1)k \end{bmatrix}$};
\node (34) [right=of 33] {};
\node (35) [right=of 34] {};
\node (36) [right=of 35] {$\begin{bmatrix} 345 \\ (k-2)(k-1)k \end{bmatrix}$};
\node (37) [right=of 36] {$\begin{bmatrix} 234 \\ (k-2)(k-1)k \end{bmatrix}$};
\node (38) [right=of 37,rectangle,draw=black,thick] {$\begin{bmatrix} 123 \\ (k-2)(k-1)k \end{bmatrix}$};

\node (41) [below=of 31] {};
\node (42) [right=of 41] {};
\node (43) [right=of 42] {};
\node (44) [right=of 43] {};
\node (45) [right=of 44] {};
\node (46) [right=of 45] {};
\node (47) [right=of 46] {};
\node (48) [right=of 47] {};

\node (51) [below=of 41] {};
\node (52) [right=of 51] {};
\node (53) [right=of 52] {};
\node (54) [right=of 53] {};
\node (55) [right=of 54] {};
\node (56) [right=of 55] {};
\node (57) [right=of 56] {};
\node (58) [right=of 57] {};

\node (61) [below=of 51] {$(2j)$}; 
\node (62) [right=of 61] {$\begin{bmatrix} (j-1)j \\ 23 \end{bmatrix}$};
\node (63) [right=of 62] {$\begin{bmatrix} (j-2)(j-1)j \\ 234 \end{bmatrix}$};
\node (64) [right=of 63] {};
\node (65) [right=of 64] {};
\node (66) [right=of 65] {$\begin{bmatrix} 34\cdots (k+1) \\ 23\cdots k \end{bmatrix}$};
\node (67) [right=of 66] {$\begin{bmatrix} 23\cdots k \\ 23\cdots k \end{bmatrix}$};
\node (68) [right=of 67,rectangle,draw=black,thick] {$\begin{bmatrix} 12\cdots (k-1) \\ 23\cdots k \end{bmatrix}$};

\node (71) [below=of 61,rectangle,draw=black,thick] {$(1j)$}; 
\node (72) [right=of 71,rectangle,draw=black,thick] {$\begin{bmatrix} (j-1)j \\ 12 \end{bmatrix}$};
\node (73) [right=of 72,rectangle,draw=black,thick] {$\begin{bmatrix} (j-2)(j-1)j \\ 123 \end{bmatrix}$};
\node (74) [right=of 73] {};
\node (75) [right=of 74] {};
\node (76) [right=of 75,rectangle,draw=black,thick] {$\begin{bmatrix} 34\cdots (k+2) \\ 12\cdots k \end{bmatrix}$};
\node (77) [right=of 76,rectangle,draw=black,thick] {$\begin{bmatrix} 23\cdots (k+1) \\ 12\cdots k \end{bmatrix}$};
\node (78) [right=of 77,rectangle,draw=black,thick] {$\begin{bmatrix} 12\cdots k \\ 12\cdots k \end{bmatrix}$};

\draw[semithick,->] (11) to (12);
\draw[semithick,->] (12) to (13);
\draw[semithick,->] (13) to (14);
\draw[semithick,dotted] (14) to (15);
\draw[semithick,->] (15) to (16);
\draw[semithick,->] (16) to (17);
\draw[semithick,->] (17) to (18);

\draw[semithick,->] (21) to (22);
\draw[semithick,->] (22) to (23);
\draw[semithick,->] (23) to (24);
\draw[semithick,dotted] (24) to (25);
\draw[semithick,->] (25) to (26);
\draw[semithick,->] (26) to (27);
\draw[semithick,->] (27) to (28);

\draw[semithick,->] (31) to (32);
\draw[semithick,->] (32) to (33);
\draw[semithick,->] (33) to (34);
\draw[semithick,dotted] (34) to (35);
\draw[semithick,->] (35) to (36);
\draw[semithick,->] (36) to (37);
\draw[semithick,->] (37) to (38);

\draw[semithick,dotted] (48) to (58);

\draw[semithick,->] (61) to (62);
\draw[semithick,->] (62) to (63);
\draw[semithick,->] (63) to (64);
\draw[semithick,dotted] (64) to (65);
\draw[semithick,->] (65) to (66);
\draw[semithick,->] (66) to (67);
\draw[semithick,->] (67) to (68);

\draw[semithick,dotted] (74) to (75);

\draw[semithick,->] (11) to (21);
\draw[semithick,->] (21) to (31);
\draw[semithick,->] (31) to (41);
\draw[semithick,dotted] (41) to (51);
\draw[semithick,->] (51) to (61);
\draw[semithick,->] (61) to (71);

\draw[semithick,->] (12) to (22);
\draw[semithick,->] (22) to (32);
\draw[semithick,->] (32) to (42);
\draw[semithick,dotted] (42) to (52);
\draw[semithick,->] (52) to (62);
\draw[semithick,->] (62) to (72);

\draw[semithick,->] (13) to (23);
\draw[semithick,->] (23) to (33);
\draw[semithick,->] (33) to (43);
\draw[semithick,dotted] (43) to (53);
\draw[semithick,->] (53) to (63);
\draw[semithick,->] (63) to (73);

\draw[semithick,->] (16) to (26);
\draw[semithick,->] (26) to (36);
\draw[semithick,->] (36) to (46);
\draw[semithick,dotted] (46) to (56);
\draw[semithick,->] (56) to (66);
\draw[semithick,->] (66) to (76);

\draw[semithick,->] (17) to (27);
\draw[semithick,->] (27) to (37);
\draw[semithick,->] (37) to (47);
\draw[semithick,dotted] (47) to (57);
\draw[semithick,->] (57) to (67);
\draw[semithick,->] (67) to (77);

\draw[semithick,->] (22) to (11);
\draw[semithick,->] (23) to (12);
\draw[semithick,->] (24) to (13);
\draw[semithick,->] (26) to (15);
\draw[semithick,->] (27) to (16);
\draw[semithick,->] (28) to (17);

\draw[semithick,->] (32) to (21);
\draw[semithick,->] (33) to (22);
\draw[semithick,->] (34) to (23);
\draw[semithick,->] (36) to (25);
\draw[semithick,->] (37) to (26);
\draw[semithick,->] (38) to (27);

\draw[semithick,->] (42) to (31);
\draw[semithick,->] (43) to (32);
\draw[semithick,->] (44) to (33);
\draw[semithick,->] (47) to (36);
\draw[semithick,->] (48) to (37);

\draw[semithick,->] (62) to (51);
\draw[semithick,->] (63) to (52);
\draw[semithick,->] (66) to (55);
\draw[semithick,->] (67) to (56);
\draw[semithick,->] (68) to (57);

\draw[semithick,->] (72) to (61);
\draw[semithick,->] (73) to (62);
\draw[semithick,->] (74) to (63);
\draw[semithick,->] (76) to (65);
\draw[semithick,->] (77) to (66);
\draw[semithick,->] (78) to (67);

\end{tikzpicture}}
}
\end{center}
\caption{Initial cluster for a quantum cluster algebra structure on $\KqMat{k}{j}$.\label{fig:initialseed}}
\end{figure}
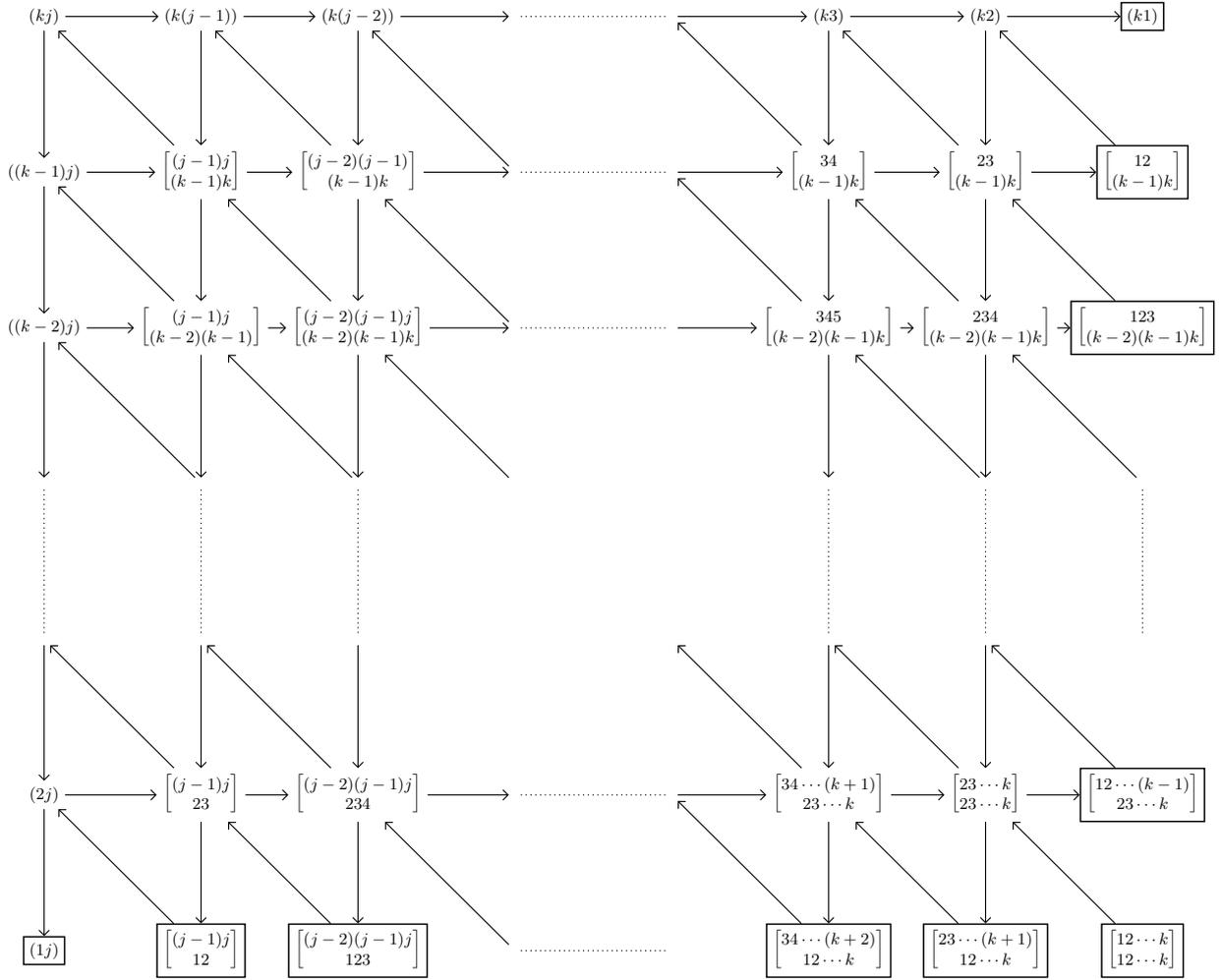

\begin{figure}[t]
\begin{center}
{\footnotesize
\scalebox{1}{\begin{tikzpicture}[node distance=2cm,on grid,>=angle 90]

\node (11) at (0,0) {$(33)$}; 
\node (12) [right=of 11] {$(32)$};
\node (13) [right=of 12,rectangle,draw=black,thick] {$(31)$};

\node (21) [below=of 11] {$(23)$}; 
\node (22) [right=of 21] {$\begin{bmatrix} 23 \\ 23 \end{bmatrix}$};
\node (23) [right=of 22,rectangle,draw=black,thick] {$\begin{bmatrix} 12 \\ 23 \end{bmatrix}$};

\node (31) [below=of 21,rectangle,draw=black,thick] {$(13)$}; 
\node (32) [right=of 31,rectangle,draw=black,thick] {$\begin{bmatrix} 23 \\ 12 \end{bmatrix}$};
\node (33) [right=of 32,rectangle,draw=black,thick] {$\begin{bmatrix} 123 \\ 123 \end{bmatrix}$};

\draw[semithick,->] (11) to (12);
\draw[semithick,->] (12) to (13);

\draw[semithick,->] (21) to (22);
\draw[semithick,->] (22) to (23);

\draw[semithick,->] (11) to (21);
\draw[semithick,->] (21) to (31);

\draw[semithick,->] (12) to (22);
\draw[semithick,->] (22) to (32);

\draw[semithick,->] (22) to (11);
\draw[semithick,->] (23) to (12);

\draw[semithick,->] (32) to (21);
\draw[semithick,->] (33) to (22);

\end{tikzpicture}}
}
\end{center}
\caption{\label{fig:initialseed33}Initial cluster for a quantum cluster algebra structure on $\KqMat{3}{3}$.}
\end{figure}

As is usual, vertices of the quiver that are frozen, \ie corresponding to elements that are coefficients and so not mutated, are boxed (a so-called ice quiver).  We do not record here explicitly the quasi-commutation matrix.  We denote this initial data for the quantum cluster algebra structure on $\KqMat{k}{j}$ as $(\curly{M}(k,j),B(k,j),L(k,j))$, where $\curly{M}(k,j)$ is the initial cluster as above and $B$ and $L$ are the exchange and quasi-commutation matrices.

Then the main theorem of \cite{GLS-QuantumPFV}, Theorem~12.3, tells us that in this case, with the above initial data, $\KqMat{k}{j}$ is a quantum cluster algebra.  We note particularly Corollary~12.4 of \cite{GLS-QuantumPFV} which says that every relevant unipotent quantum minor occurs as a quantum cluster variable in this quantum cluster algebra structure.  That is, in our particular case, every quantum minor in $\KqMat{k}{j}$ does indeed occur as a quantum cluster variable.  Of course, outside the finite-type cases, we must have quantum cluster variables that are not quantum minors; we will say a little more about these below.

We observe that this quantum cluster algebra structure can be considered as a graded quantum cluster algebra structure, with respect to the natural choice of grading.  We have that $\KqMat{k}{j}$ is an $\nat$-graded algebra when we put all the matrix generators $X_{ij}$ in degree 1.  Indeed, our choice of initial seed consists of homogeneous elements for this grading, as follows.

\begin{lemma}\label{lemma:degofMrs} We have $\card{R(r,s)}=\card{C(r,s)}=\min(r,s)$, and so $\mathrm{deg}(\curly{M}_{kj}(r,s))=\min(r,s)$. \qed
\end{lemma}

So we set $G(k,j)$ to be the vector whose $(r,s)$-entry is equal to $\min(r,s)$.  Furthermore, we see in the next lemma that the exchange quiver satisfies the required homogeneity property with respect to this grading.

\begin{lemma}\label{lemma:degsumsequal} At any mutable vertex $(\alpha,\beta)$,  \[ \sum_{(\gamma,\delta)\to (\alpha,\beta)} \deg(\curly{M}_{kj}(\gamma,\delta))=\sum_{(\alpha,\beta)\to (\gamma,\delta)} \deg(\curly{M}_{kj}(\gamma,\delta)). \]
\end{lemma}

\begin{proof}
For $\alpha=\beta=1$, we see that the two sums are equal to $2$.  

Next assume that $\alpha=1$ and $\beta>1$.  Then the vertices with arrows incoming to $(1,\beta)$ are $(1,\beta-1)$ and $(2,\beta+1)$, and the vertices with arrows outgoing from $(1,\beta)$ are $(1,\beta+1)$ and $(2,\beta)$.  Since $\deg(\curly{M}_{kj}(r,s))=\min(r,s)$, we see that the two sums are both equal to $1+2=3$, as $\beta>1$.  Similarly the sums are equal (and equal to $3$) if $\alpha>1$ and $\beta=1$.

If $\alpha,\beta >1$, then $(\alpha,\beta)$ has six neighbours: $(\alpha,\beta-1)$, $(\alpha-1,\beta)$ and $(\alpha+1,\beta+1)$ with incoming arrows, and $(\alpha-1,\beta-1)$, $(\alpha,\beta+1)$ and $(\alpha+1,\beta)$ outgoing.  Then if $\alpha=\beta$, the two sums are easily seen to be equal to $3\alpha-1$, or if $\alpha<\beta$ the sums are equal to $3\alpha$, or if $\beta<\alpha$ they are equal to $3\beta$.
\end{proof}

Then by our earlier discussion, the quantum cluster algebra associated to the initial quantum seed $(\curly{M}(k,j),B(k,j),L(k,j),G(k,j))$ is $\integ$-graded and in particular every quantum cluster variable is homogeneous with respect to this grading.  Note that \emph{a priori} we only deduce a $\integ$-grading.

\vfill
\pagebreak
This grading also has a categorical interpretation.  As described in \S 9.6 of \cite{GLS-QuantumPFV}, drawing on \cite[{\S 10}]{GLS-KacMoody}, every module $X$ in $\curly{C}_{w}$ has a filtration
\[ 0=X_{0} \subseteq X_{1} \subseteq X_{2} \subseteq \cdots \subseteq X_{r}=X \] such that each subquotient $X_{i}/X_{i-1}$ is isomorphic to $M_{i}^{m_{i}}$ where $M_{i}$ is the module $M_{(k-\alpha+1,j-\beta+1)}$ corresponding to $V_{i}=V_{(\alpha,\beta)}$ (where $V_{k}$ was our original numbering for the $V$'s, coming from $\underline{i}$).  Hence each module in $\curly{C}_{w}$ has an $M$-dimension vector, $\underline{m}(X)=(m_{1},\dotsc ,m_{r})\in \nat^{r}$.  

The general theory tells us that the modules $V_{(\alpha,\beta)}$ can be considered as being built up by repeated extensions of the modules $M_{(a,b)}$ (corresponding to the algebraic identification of $V_{(\alpha,\beta)}$ being a minor and thus a product of the matrix generators, to which the $M_{(a,b)}$ correspond).  In the case at hand, we see that the minor corresponding to $V_{(\alpha,\beta)}$ is an $l\cross l$ minor exactly when $V_{(\alpha,\beta)}$ has $\sum_{i=1}^{r} \underline{m}(V_{(\alpha,\beta)})_{i}=l$, i.e. when $V_{(\alpha,\beta)}$ has $l$ non-zero subquotients of the form described in the previous paragraph.  We call this the sum of the entries of the $M$-dimension vector the $M$-dimension of the module.

We see that the initial exchange matrix (or quiver) has the necessary property to imply that this gives a grading by looking at the explicit description of the arrows.  For example, the arrow $(\alpha,\beta)\to (\alpha-1,\beta-1)$ exactly corresponds to the inclusion homomorphism $V_{(\alpha-1,\beta-1)}\to V_{(\alpha,\beta)}$ for which $M_{(k-\alpha+1,j-\beta+1)}$ is the cokernel, thus $V_{(\alpha,\beta)}$ has $M$-dimension one greater than that of $V_{(\alpha-1,\beta-1)}$.  One see that in the grid arrangement, $M$-dimension is constant along rows and increases by one on going down a row.  Away from the boundary, every mutable vertex has the same number of arrows coming in from a given row as going out to it (either zero or one of each, in fact) and so the two sums of $M$-dimensions over arrows entering or leaving the vertex are equal.  It is straightforward to check that the boundary cases also have the required property.

Indeed the fact that every module in $\curly{C}_{w}$ has a filtration with subquotients the modules $M_{(a,b)}$ makes it clear that this grading is the usual $\nat$-grading on $\KqMat{k}{j}$, for we have these modules $M_{(a,b)}$ in degree one as for the matrix generators.  Thus we can view the above statement as saying that the quantum cluster algebra structure is compatible with the natural graded algebra structure of $\KqMat{k}{j}$.  This will be important for us later.  Again, we see that this is a property of the category that does not rely on being in the quantum case: this grading is present whether one considers $\curly{C}_{w}$ to be categorifying the commutative or the quantum coordinate ring.

Finally,\label{page:theta-start} we note one more grading-like datum associated to the category $\curly{C}_{w}$.  Namely, following \cite[\S 10]{GLS-PFV}, to each module $M$ in $\curly{C}_{w}$ we may associate the natural number given by $\theta(M)=\dim \mathrm{Hom}_{\Lambda}(M,S_{j})$, where $j=m-k+1$.  Then $\theta(M)$ is the multiplicity of $S_{k}$ in the top of the module $M$ and we see from the above that $\theta(V_{(\alpha,\beta)})=\theta(M_{(\alpha,\beta)})=1$ for all $1\leq \alpha \leq k$, $1\leq \beta \leq j$.  However $\theta$ is not always equal to 1: in the example of $\KqMat{3}{3}$, mutating $V_{(2,2)}$ yields a module $W$ with $\theta(W)=2$.

This data has the property that it is compatible with mutation, in the following sense: if $M^{\prime}$ is the module obtained by mutating $M$, so that there exist two exact sequences
\[ 0 \to M \to U \to M^{\prime} \to 0 \qquad \text{and} \qquad 0 \to M^{\prime} \to W \to M \to 0 \]
where $U$ and $W$ correspond to the exchange monomials, then
\[ \dim \mathrm{Hom}_{\Lambda}(M^{\prime},S_{j})=\max \{ \dim \mathrm{Hom}_{\Lambda}(U,S_{j}), \dim \mathrm{Hom}_{\Lambda}(W,S_{j})\}-\dim \mathrm{Hom}_{\Lambda}(M,S_{j}). \]  

\begin{remark} An analogue of this formula is stated as \cite[Proposition~10.1]{GLS-PFV}, for socles as opposed to tops.  The paper \cite{GLS-PFV}, in which the classical version of the topic of this paper is considered, works with a category of submodules whereas the quantum version in \cite{GLS-QuantumPFV} uses a category of factor modules.  Consequently, to fit with \cite{GLS-QuantumPFV} we need to look at tops here as opposed to socles.
\end{remark}

More compactly, in the above notation, $\theta(M^{\prime})=\max\{ \theta(U),\theta(W) \}-\theta(M)$.  That is, given the values of $\theta$ on a collection of modules associated to an initial cluster, one may calculate the values on all modules associated to cluster variables.  Notice that because $\theta$ is a dimension, it necessarily takes natural number values; that the formula $\theta(M^{\prime})=\max\{ \theta(U),\theta(W) \}-\theta(M)$ produces this is not \emph{a priori} clear.

Note also that this data is not a grading for the (quantum) cluster algebra structure above.  Indeed, at the vertex indexed by $(1,1)$ we have two outgoing arrows to modules each of which has a 1-dimensional top but only one incoming arrow, from a module that also has a 1-dimensional top.  That is, $\theta(U)\neq \theta(W)$ in this case, although the formula does tell us that the mutated module also has a 1-dimensional top.  At all other mutable vertices for the cluster $\curly{M}(k,j)$ we do have homogeneity with respect to this function $\theta$, however.\label{page:theta-end}

\section{The dehomogenisation isomorphism and the image of the cluster structure under this}\label{s:dehomog}

In work of Kelly, Lenagan and Rigal (\cite{KellyLenaganRigal}), a noncommutative dehomogenisation of an $\nat$-graded algebra is defined and their Corollary~4.1 describes an isomorphism of the localisation of the quantum Grassmannian at the minor $[(n-k+1)\dotsm n]$ with a skew-Laurent extension of a quantum matrix algebra.  In \cite{LenaganRussell}, a dehomogenisation isomorphism $\rho$ involving $\KqGr{k}{n}$ localised at the consecutive minor $[\widetilde{a}\ (\widetilde{a+1})\ \cdots\ (\widetilde{a+k-1})]$ is constructed, where ``$\,\widetilde{\ \ }\,$'' indicates that values are taken modulo $n$ and from the set $\{1,\ldots,n\}$.  In order to match conventions already fixed, we will need the map corresponding to the special case of the map $\rho$ of \cite{LenaganRussell} for the value $a=1$, the original work of \cite{KellyLenaganRigal} being the case $a=n-k+1$.

This map is key to the lifting procedure to obtain the quantum cluster algebra structure on the quantum Grassmannian, and we recall its definition.  (Here, $\widehat{a}$ denotes an omitted index.)

\begin{proposition}[\cite{LenaganRussell}]\label{p:dhom} Let $\sigma$ be the automorphism of $\KqMat{k}{n-k}$ defined by $\sigma(X_{ij})=qX_{ij}$.  The map
\[ \alpha\colon \KqMat{k}{n-k}[Y^{\pm 1}; \sigma] \to \KqGr{k}{n}\!\left[ [12\cdots k]^{-1} \right] \]
defined by
\[ \alpha(X_{ij}) = [1\cdots \widehat{k-i+1} \cdots k\ (j+k)][1\cdots k]^{-1}, \quad \alpha(Y) =[12\cdots k] \]
is an algebra isomorphism. \qed
\end{proposition}

This map allows us to transport the quantum cluster algebra structure on quantum matrices, as described in the previous section, to the above localisation of the quantum Grassmannian.  Set $\mathrm{Loc}(\KqGr{k}{n})=\KqGr{k}{n}\!\left[ [12\cdots k]^{-1} \right]$.  

From generalities on noncommutative dehomogenisations (see \cite[\S 3]{KellyLenaganRigal}), the $\nat$-grading on $\KqGr{k}{n}$ that has all the generating Pl\"{u}cker coordinates $[I]$ in degree one gives rise to a $\integ$-grading on the localisation.  Let $\KqGr{k}{n}_{i}$ denote the degree $i$ homogeneous component of $\KqGr{k}{n}$ for the aforementioned grading.  Then
\[ \mathrm{Loc}(\KqGr{k}{n})=\bigdsum_{l\in \integ} \mathrm{Loc}(\KqGr{k}{n})_{l} \]
with
\[ \mathrm{Loc}(\KqGr{k}{n})_{l}=\sum_{j\geq 0} \KqGr{k}{n}_{l+j}[1\dotsm k]^{-j}. \]
The noncommutative dehomogenisation of $\KqGr{k}{n}$ is defined to be the degree 0 part of $\mathrm{Loc}(\KqGr{k}{n})$ with respect to this $\integ$-grading and the results of \cite{KellyLenaganRigal} and \cite{LenaganRussell} show that this degree 0 part is isomorphic to the quantum matrices $\KqMat{k}{n-k}$, via the map $\alpha$ of Proposition~\ref{p:dhom}.  In particular, the map $\alpha$ sends an element of the quantum matrices to a $\mathbb{K}$-linear combination of elements of the form $m[1\dotsm k]^{-j}$ where $m$ is a homogeneous polynomial of degree $j$ in $\KqGr{k}{n}$.  

We note that $\KqGr{k}{n}$ is a subalgebra of $\mathrm{Loc}(\KqGr{k}{n})$ but that $\KqGr{k}{n}$ has non-trivial intersection with every homogeneous component of $\mathrm{Loc}(\KqGr{k}{n})$.  More precisely, $\KqGr{k}{n}_{l}\subseteq \mathrm{Loc}(\KqGr{k}{n})_{l}$ for every $l$; less formally, $\KqGr{k}{n}$ is ``spread out'' across all the components of $\mathrm{Loc}(\KqGr{k}{n})$ and this is at the root of the technical difficulties that must be overcome in order to deduce our main result.

Now Proposition~\ref{p:dhom} describes the algebra $\mathrm{Loc}(\KqGr{k}{n})$ in terms of a skew-Laurent extension of the quantum matrices and so we can use Proposition~\ref{p:skewLaurent-extra-coeffs} to deduce the following.

\begin{proposition} The localisation $\mathrm{Loc}(\KqGr{k}{n})$ is a graded quantum cluster algebra.
\end{proposition}

\begin{proof} We have seen in Section~\ref{s:KqMatisQCA} that $\KqMat{k}{n-k}$ is a graded quantum cluster algebra with the initial data provided by the theorem of Gei\ss, Leclerc and Schr\"{o}er and the grading being the standard grading on $\KqMat{k}{n-k}$.  In particular, the initial cluster consists of homogeneous elements and the automorphism $\sigma$ in Proposition~\ref{p:dhom} therefore acts on the initial cluster variable $\curly{M}_{k(n-k)}(r,s)=\minor{R(r,s)}{C(r,s)}$ by multiplication by $q^{\deg(\curly{M}_{k(n-k)}(r,s))}=q^{\min(r,s)}$.  Thus the required conditions for applying Proposition~\ref{p:skewLaurent-extra-coeffs} hold and the skew-Laurent extension $\KqMat{k}{n-k}[Y^{\pm 1};\sigma]$ induced by $\sigma$ is a graded quantum cluster algebra.  Note that we have as additional coefficients $Y$ and $Y^{-1}$.

Since the map $\alpha$ in Proposition~\ref{p:dhom} is an algebra isomorphism, this structure is transported to $\mathrm{Loc}(\KqGr{k}{n})$.  The extra coefficients $Y$ and $Y^{-1}$ are mapped under $\alpha$ to $[1\dotsm k]$ and $[1\dotsm k]^{-1}$ respectively.  We choose to place the coefficient $[1\dotsm k]$ at position $(1,k+1)$ and $[1\dotsm k]^{-1}$ at $(2,k+1)$; these choices are arbitrary.  

We will denote by $\curly{L}(k,n)$ the union of the image of $\curly{M}(k,n-k)$ under $\alpha$ with the set $\{ [1\dotsm k],[1\dotsm k]^{-1} \}$.  The corresponding exchange matrix will be denoted $B^{\mathrm{Loc}}(k,n)$; this matrix is obtained from the exchange matrix $B(k,n-k)$ by adding rows and columns indexed as $(1,k+1)$ and $(2,k+1)$ consisting of zeroes, corresponding to the two extra coefficients.  The quasi-commutation matrix will be denoted $L^{\mathrm{Loc}}(k,n)$ and, as described in Proposition~\ref{p:skewLaurent-extra-coeffs}, this is determined by $L(k,n-k)$ and the automorphism $\sigma$.  

For our grading, described by $G^{\mathrm{Loc}}(k,n) \in \integ^{k(n-k)+2}$, we take the corresponding entry from $G(k,n-k)$ for elements of the image of $\curly{M}(k,n-k)$ under $\alpha$ and take $G^{\mathrm{Loc}}_{(1,k+1)}=1$ and $G^{\mathrm{Loc}}_{(2,k+1)}=-1$, in accordance with the natural choice.  We note that this is not the natural grading on $\mathrm{Loc}(\KqGr{k}{n})$ described above: that would of course have the image of $\curly{M}(k,n-k)$ in degree 0, though it would agree with $G^{\mathrm{Loc}}_{(1,k+1)}=1$ and $G^{\mathrm{Loc}}_{(2,k+1)}=-1$.  In making the choice of $G^{\mathrm{Loc}}(k,n)$ described here we are explicitly choosing to retain the grading associated to the pre-image under $\alpha$, in $\KqMat{k}{n-k}$.

Thus $\mathrm{Loc}(\KqGr{k}{n})$ has a quantum cluster algebra structure with initial data
\[ (\curly{L}(k,n),B^{\mathrm{Loc}}(k,n),L^{\mathrm{Loc}}(k,n),G^{\mathrm{Loc}}(k,n)). \qedhere \]
\end{proof}

However, this does not show that $\KqGr{k}{n}$ is a quantum cluster algebra.  We see from Proposition~\ref{p:dhom} that $\alpha$ maps each matrix generator $X_{ij}$ to the product of a quantum Pl\"{u}cker coordinate and the element $[1\cdots k]^{-1}$, and these are not elements of $\KqGr{k}{n}$, viewed as a subalgebra of the localisation in the obvious way.  So although $\KqGr{k}{n}$ is a subalgebra of a quantum cluster algebra, it is not a cluster subalgebra: the cluster variables are not elements of the subalgebra.

For later use, we record a special case of a companion result of Lenagan and Russell (\cite[Proposition~3.3]{LenaganRussell}) that describes the image of a quantum minor in $\KqMat{k}{n-k}$ under the map $\alpha$.  

First, for indexing sets $I=\{ i_{1},\ldots, i_{t}\}$ and $J=\{ j_{1},\ldots ,j_{t}\}$, define 
\[ Q_{1}(I,J) =\{ \widetilde{j_{1}+k},\widetilde{j_{2}+k},\ldots,\widetilde{j_{t}+k}\} \disjointunion \left( \{ 1,\ldots,k \} \setminus \{ k-i_{1}+1,k-i_{2}+1,\ldots k-i_{t}+1 \} \right) \]
It is straightforward to verify that this is a subset of $\{1,\ldots, n\}$ of cardinality $k$.  

\begin{lemma}[{\cite{LenaganRussell}}]\label{l:alphaofminor} Let $\minor{I}{J}$ denote the quantum minor in $\KqMat{k}{n-k}$ with row and column indexing sets $I$ and $J$.  Then
\[ \alpha(\minor{I}{J})=[Q_{1}(I,J)][1\cdots k]^{-1}.  \] \qed 
\end{lemma}
Note that the sets $Q_{1}(I,J)$ describe column sets of maximal minors, whose row set is necessarily $\{1,\dotsc,k\}$. 

This lemma has the following two consequences for the quantum cluster algebra structure on $\mathrm{Loc}(\KqGr{k}{n})$.  First, we see that we can rephrase the lemma in terms of the components of the noncommutative dehomogenisation as described above.  Recall that $\alpha$ restricts to an isomorphism of $\KqMat{k}{n-k}$ with the degree 0 part of $\mathrm{Loc}(\KqGr{k}{n})$; the lemma tells us more.

\begin{corollary}\label{c:imageofICunderalpha} The image of $\curly{M}(k,n-k)$ under $\alpha$ in $\mathrm{Loc}(\KqGr{k}{n})$ is contained in the subspace $\KqGr{k}{n}_{1}[1\cdots k]^{-1}$ of $\mathrm{Loc}(\KqGr{k}{n})_{0}=\sum_{j\geq 0} \KqGr{k}{n}_{j}[1\dotsm k]^{-j}$.  \qed
\end{corollary}

From Definition~\ref{d:qGLScluster}, we may compute the image of $\curly{M}(k,n-k)$ under $\alpha$ explicitly.

\begin{lemma} Let $\minor{R(r,s)}{C(r,s)} \in \curly{M}(k,n-k)$.  Then
\[ \alpha(\minor{R(r,s)}{C(r,s)}) = \left[\{ \widetilde{1-s},\widetilde{2-s},\ldots ,\widetilde{r-s} \} \disjointunion \left(\{ 1,\ldots,k\} \setminus \{r,\ldots ,r+s\}\right)\right][1\cdots k]^{-1}. \] \qed
\end{lemma}
\noindent We illustrate this for our running example in Figure~\ref{fig:initialseedLoc36}.

\pagebreak
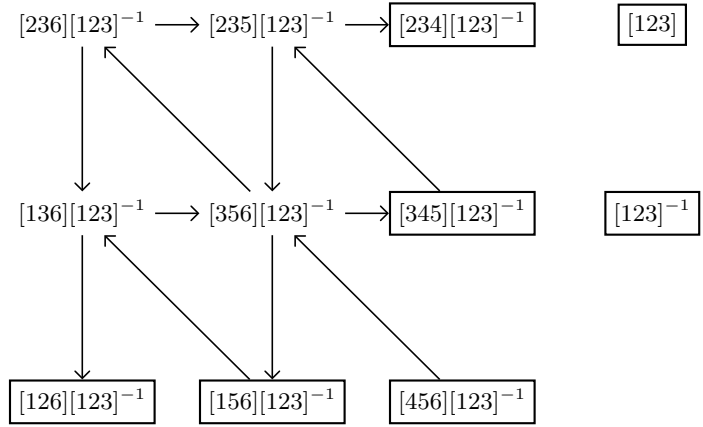
\begin{figure}[t]
\begin{center}
{\footnotesize
\scalebox{1}{\begin{tikzpicture}[node distance=2.5cm,on grid,>=angle 90]

\node (11) at (0,0) {$[236][123]^{-1}$}; 
\node (12) [right=of 11] {$[235][123]^{-1}$};
\node (13) [right=of 12,rectangle,draw=black,thick] {$[234][123]^{-1}$};

\node (21) [below=of 11] {$[136][123]^{-1}$}; 
\node (22) [right=of 21] {$[356][123]^{-1}$};
\node (23) [right=of 22,rectangle,draw=black,thick] {$[345][123]^{-1}$};

\node (31) [below=of 21,rectangle,draw=black,thick] {$[126][123]^{-1}$}; 
\node (32) [right=of 31,rectangle,draw=black,thick] {$[156][123]^{-1}$};
\node (33) [right=of 32,rectangle,draw=black,thick] {$[456][123]^{-1}$};

\node (14) [right=of 13,rectangle,draw=black,thick] {$[123]$};
\node (24) [right=of 23,rectangle,draw=black,thick] {$[123]^{-1}$};

\draw[semithick,->] (11) to (12);
\draw[semithick,->] (12) to (13);

\draw[semithick,->] (21) to (22);
\draw[semithick,->] (22) to (23);

\draw[semithick,->] (11) to (21);
\draw[semithick,->] (21) to (31);

\draw[semithick,->] (12) to (22);
\draw[semithick,->] (22) to (32);

\draw[semithick,->] (22) to (11);
\draw[semithick,->] (23) to (12);

\draw[semithick,->] (32) to (21);
\draw[semithick,->] (33) to (22);

\end{tikzpicture}}
}
\end{center}
\caption{\label{fig:initialseedLoc36}Initial cluster for a quantum cluster algebra structure on $\mathrm{Loc}(\KqGr{3}{6})$.}
\end{figure}

These results tell us about our initial cluster variables but in order to complete the lifting of the whole quantum cluster algebra structure to the quantum Grassmannian, we need a stronger statement on the images of \emph{all} quantum cluster variables.  This is achieved by the following theorem, which uses a cluster algebra argument, as opposed to direct calculation of the sort that gives the above results.  It also emphasises the relevance of the categorification.

\begin{theorem}\label{t:thetaisscalingpower} Let $v$ be a quantum cluster variable for the quantum cluster algebra structure on $\KqMat{k}{n-k}$ constructed from the initial data $\curly{M}(k,n-k)$.  Let $M(v)$ be the module in $\curly{C}_{w}$ corresponding to $v$.  Then $\alpha(v) \in \KqGr{k}{n}_{\theta(M(v))}[1\dotsm k]^{-\theta(M(v))}$, where $\theta(M(v))$ is equal to the dimension of the top of the module $M(v)$.
\end{theorem}

\begin{proof} We argue by induction on the length of mutation sequences.  Firstly, we see that the claim holds for elements of the initial cluster $\curly{M}(k,n-k)$ by the observation that $\theta(V_{(\alpha,\beta)})=1$ for all $\alpha$ and $\beta$ and by Corollary~\ref{c:imageofICunderalpha}.  

Note that the localisation at hand, $\mathrm{Loc}(\KqGr{k}{n})$, is constructed from $\KqGr{k}{n}$ by localisation at an Ore set that consists of (positive integer) powers of a single element, $[1\cdots k]$.  Since in any localisation a finite set of elements has a common denominator, for any finite set of elements $A_{1},\ldots,A_{r}$ of $\mathrm{Loc}(\KqGr{k}{n})$ there exists a positive integer $m$ such that there exist $B_{1},\ldots,B_{r} \in \KqGr{k}{n}$ with $A_{i}=B_{i}[1\cdots k]^{-m}$.  Furthermore, we can choose $m$ to be the (unique) least such positive integer, so that we may speak of lowest common denominators in this localisation.  

In particular, this holds for $r=1$, so that every element $A$ of $\mathrm{Loc}(\KqGr{k}{n})$ has a unique expression as $B[1\cdots k]^{-d(A)}$ with $B\in \KqGr{k}{n}$ and $d(A)$ minimal.  Equivalently, there exists a unique smallest $d(A)$ such that $A$ is an element of the subspace $\KqGr{k}{n}_{d(A)}[1\cdots k]^{-d(A)}$ of $\mathrm{Loc}(\KqGr{k}{n})$.  

Since $[1\cdots k]\in \KqGr{k}{n}$, we see that this implies that $A\in \KqGr{k}{n}_{j}[1\cdots k]^{-j}$ for all $j\geq d(A)$ since \[ A=B[1\cdots k]^{-d(A)}=(B[1\cdots k]^{j-d(A)})[1\cdots k]^{-j}. \] This is why the decomposition $\mathrm{Loc}(\KqGr{k}{n})=\sum_{j\geq 0} \KqGr{k}{n}_{j}[1\cdots k]^{-j}$ is not a direct sum decomposition.

\pagebreak
Now, assume that the claim holds for some cluster $\curly{N}$ mutation-equivalent to $\curly{M}(k,n-k)$.  Let $X_{1},\dotsc ,X_{r}$ be the quantum cluster variables appearing in $\curly{N}$.  Then the mutation of $X_{i}$, say, is computed by taking the sum of the two relevant exchange monomials.  More precisely, recall that the exchange relations take the form 
\[ X_{i}^{\prime}=M(\underline{b}_{i}^{+})+M(\underline{b}_{i}^{-}) \]
with
\[ M(a_{1},\dotsc ,a_{r}) \defeq q^{\frac{1}{2}\sum_{u<v} a_{u}a_{v}l_{vu}}X_{1}^{a_{1}}\dotsm X_{r}^{a_{r}} \]
and the integers $a_{i}$ are all non-negative except for $a_{k}=-1$.  It is convenient to observe, however, that since $X_{i}$ quasi-commutes with every other element of the cluster, so does its inverse and we can re-write the exchange relation in the form
\[ X_{i}^{\prime}X_{i}=N_{+}+N_{-} \]
by quasi-commuting $X_{i}^{-1}$ to the right-hand side of each monomial $M(\underline{b}_{i}^{+})$ and $M(\underline{b}_{i}^{-})$ and multiplying through.  Of course, this changes the powers of $q$ appearing in front of the monomials but for the present argument this does not matter.  Now we see that we may use the inductive hypo\-thesis and the existence of lowest common denominators to first write $N_{\pm}=S_{\pm}[1\cdots k]^{-d(N_{\pm})}$ with $S_{\pm}\in \KqGr{k}{n}$ and $d(N_{\pm})$ positive integers and then to write $N_{+}+N_{-}$ as a product of an element of $\KqGr{k}{n}$ with some power $m$ of $[1\cdots k]^{-1}$.  Furthermore it is clear that the minimal such $m$ is equal to the maximum of the powers $d(N_{+})$ and $d(N_{-})$.  Hence the exchange relation tells us that when we write $X_{i}=T_{i}[1\cdots k]^{-d(X_{i})}$ and  $X_{i}^{\prime}=T^{\prime}_{i}[1\cdots k]^{-d(X^{\prime}_{i})}$ with $T_{i},T_{i}^{\prime}\in \KqGr{k}{n}$, $d(X^{\prime}_{i})$ is equal to $\max\{ d(N_{+}),d(N_{-})\}-d(X_{i})$.

However, we have seen this formula previously: it is precisely the formula determining the values of $\theta$ by repeated mutation.  Since the integers given by $\theta$ (the dimension of the top of the corresponding module) and $d$ (the minimal positive integer described above) take the same initial values---this being the base case for our induction---and since they mutate by identical formul\ae, we see that they agree on all quantum cluster variables.  This proves the theorem.
\end{proof}

A more direct argument can be made for the quantum cluster variables for $\KqMat{k}{n-k}$ whose image under $\alpha$ is a quantum Pl\"{u}cker coordinate multiplied by some power of $[1\cdots k]^{-1}$.  For we may observe that as we range over all possible indexing sets $I$ and $J$ of quantum minors in $\KqMat{k}{n-k}$, the collection of sets $Q_{1}(I,J)$ ranges over \emph{all} $k$-subsets of $\{1,\dotsc ,n\}$.  Then since Corollary~12.4 of \cite{GLS-QuantumPFV} tells us that every quantum minor does occur as a quantum cluster variable in the quantum cluster algebra structure for $\KqMat{k}{n-k}$, we see the following.

\begin{corollary}\label{c:PluckersAreQCVs} For each $k$-subset $I$ of $\{1,\dotsc ,n\}$, we have that $[I][1\cdots k]^{-1}$ is a quantum cluster variable in the above quantum cluster algebra structure on $\mathrm{Loc}(\KqGr{k}{n})$. \qed
\end{corollary}

\noindent Here, $[I]$ is the quantum Pl\"{u}cker coordinate corresponding to $I$.  Note that this is consistent with the above general theorem, since the corresponding modules have simple tops; the latter follows from the fact that this is true for the modules $V_{(\alpha,\beta)}$, this itself being a feature of being in type~$A$.

\pagebreak
\section{A quantum cluster algebra structure on the quantum Grassmannian}\label{s:QCAonKqGr}

The\label{page:fix-theta-start} final step is to use Theorem~\ref{t:rescaledQCA} to re-scale the quantum cluster variables appearing in the above quantum cluster algebra structure on $\mathrm{Loc}(\KqGr{k}{n})$ to eliminate the inverse of the minor $[1\cdots k]$ that appears.  By doing so, we will see that all the re-scaled quantum cluster variables in fact lie in $\KqGr{k}{n}$, which together with Corollary~\ref{c:PluckersAreQCVs} will imply that we have a (graded) quantum cluster algebra structure on $\KqGr{k}{n}$, since the Pl\"{u}cker coordinates generate $\KqGr{k}{n}$.

From the previous section, notably Theorem~\ref{t:thetaisscalingpower}, we know that the power of $[1\cdots k]$ appearing in any quantum cluster variable is exactly given by $\theta$, the dimension of the top of the corresponding module in the categorification.  We would like to apply our re-scaling theorem, Theorem~\ref{t:rescaledQCA}, but we cannot do so directly because as we noted before $\theta$ is not a grading.  Therefore our first task is to fix this.

More concretely, we will alter slightly the initial data in order to correct the inhomogeneity at the position $(1,1)$.  For we observe that at every mutable index $(\alpha,\beta)$ except the top-left, i.e.\ except at $(1,1)$, the exchange quiver has the \emph{same} number of incoming and outgoing arrows.  In other words, the exchange matrix $B=B^{\mathrm{Loc}}(k,n)$ admits a grading by the vector $\underline{a}=(1,\dotsc ,1,-1,1)\in \integ^{k(n-k)+2}$ except for at the index $(1,1)$.  We would like to use this grading $\underline{a}$ as one of our input data to Theorem~\ref{t:rescaledQCA} but we need to homogenise $B$ at $(1,1)$ in order to do so.  We note that $\underline{a}$ ends with the values $-1$ and $1$ in order to reflect the natural grading on $\mathrm{Loc}(\KqGr{k}{n})$ in terms of the power of $[1\dotsm k]^{-1}$ occurring in our expressions for the quantum cluster variables.

We observe that the grading $\underline{a}$ does correspond to $\theta$, at least away from the coefficients $[1\cdots k]^{\pm 1}$.  This is as expected, for this categorical data $\theta$ is exactly what is used in \cite[\S 10]{GLS-PFV} to make the lifting work classically.  The classical version uses quotients rather than localisations, as we must, but we can see that reinterpreting \cite[\S 10]{GLS-PFV} in terms of localisations gives rise to the analogue of what we do here.

The underlying reason for the choice we will make below, and its classical analogue in \cite[\S 10]{GLS-PFV}, is geometric.  This is explained in the discussion after Theorem~4.14 of \cite{GSV-Book}.  That theorem describes the passage from the corresponding classical (commutative) cluster algebra structure on $\mathbb{K}[\mathrm{Mat}(k,n-k)]$ to one on $\mathbb{K}[\mathrm{Gr}(k,n)]$ in more concrete terms than those used in \cite[\S 10]{GLS-PFV}, where the general result was the focus.  (We note that we do not directly rely on the classical result but do indirectly, in that results we use from \cite{GLS-QuantumPFV} rely on the existence of the classical cluster algebra structure on $\mathbb{K}[\mathrm{Mat}(k,n-k)]$.)

Hence we define a new initial datum as follows.  Let
\[ (\curly{L}=\curly{L}(k,n),B=B^{\mathrm{Loc}}(k,n),L=L^{\mathrm{Loc}}(k,n),G=G^{\mathrm{Loc}}(k,n)) \]
denote the initial data for the quantum cluster algebra structure on $\mathrm{Loc}(\KqGr{k}{n})$ described in the previous section.  Here $\curly{L}=\{ \curly{L}_{(r,s)} \mid 1\leq r\leq k,\ 1\leq s \leq n-k \} \union \{ \curly{L}_{(1,k+1)},\curly{L}_{(2,k+1)}\}$.  We add an additional coefficient (i.e. a non-mutable variable) $\curly{L}_{(0,0)}$ to the initial cluster $\curly{L}$, namely $\curly{L}_{(0,0)}=[1\dotsm k][1\dotsm k]^{-1}$.  Let us denote by $\hat{\curly{L}}$ the set $\curly{L} \union \{ \curly{L}_{(0,0)} \}$.  Of course, this additional element is simply the identity for the algebra $\mathrm{Loc}(\KqGr{k}{n})$ and as such it certainly quasi-commutes with every element of $\curly{L}$, giving that the corresponding quasi-commutation matrix $\hat{L}$ is constructed from $L$ by setting $\hat{L}_{(0,0),(0,0)}=0$, $\hat{L}_{(0,0),(r,s)}=\hat{L}_{(r,s),(0,0)}=0$ for all $(r,s)\in \{ 1\leq r \leq k, 1\leq s \leq n-k \} \union \{ (1,k+1),(2,k+1) \}$ and $\hat{L}_{(r_{1},s_{1}),(r_{2},s_{2})}=L_{(r_{1},s_{1}),(r_{2},s_{2})}$ whenever $(r_{1},s_{1}),(r_{2},s_{2})\neq (0,0)$.

Next we add an extra arrow to the exchange quiver, from $(0,0)$ to $(1,1)$, or equivalently define
\[ \hat{B}_{(r_{1},s_{1}),(r_{2},s_{2})} \defeq \begin{cases} 0 & \text{if}\ (r_{1},s_{1})=(0,0)\ \text{and}\ (r_{2},s_{2})=(0,0), \\ & \qquad \text{or}\ (r_{1},s_{1})=(0,0)\ \text{and}\ (r_{2},s_{2})\neq (1,1), \\ & \qquad \text{or}\ (r_{1},s_{1})\neq (1,1)\ \text{and}\ (r_{2},s_{2})=(0,0) \\ 1 & \text{if}\ (r_{1},s_{1})=(0,0)\ \text{and}\ (r_{2},s_{2})=(1,1) \\ -1 & \text{if}\ (r_{1},s_{1})=(1,1)\ \text{and}\ (r_{2},s_{2})=(0,0) \\ B_{(r_{1},s_{1}),(r_{2},s_{2})} & \text{otherwise.} \end{cases} \]

For a grading $\hat{G}$ we take $\hat{G}=(-1,\dotsc ,-1,1,-1)\in \integ^{k(n-k)+3}$.  Our reason for doing so is that $\hat{G}$ records the power of $[1\dotsm k]$ occurring in the initial cluster variables as shown in the previous section, namely $-1$ except for the coefficient $[1\dotsm k]$.  This is indeed still a grading, as is easily checked.  We will not forget the data in the original grading $G$: it will also be used when we apply Theorem~\ref{t:rescaledQCA}, which takes \emph{two} gradings among its inputs.

Now it is straightforward to check that $(\hat{\curly{L}},\hat{B},\hat{L},\hat{G})$ is valid initial data for a graded quantum cluster algebra structure on $\mathrm{Loc}(\KqGr{k}{n})$.  Indeed, compatibility of $\hat{B}$ with $\hat{L}$ is immediate, since $\hat{L}$ contains only zeroes in the row and column indexed by $(0,0)$.  The grading condition holds by construction.

Furthermore, the quantum cluster variables obtained by iterated mutation from the initial seed $(\hat{\curly{L}},\hat{B},\hat{L},\hat{G})$ are equal to those obtained from $(\curly{L},B,L,G)$, since the new variable is simply the identity in $\mathrm{Loc}(\KqGr{k}{n})$ and as such has no effect whatsoever on any exchange monomials it appears in.  It is to this altered graded quantum cluster algebra structure on $\mathrm{Loc}(\KqGr{k}{n})$, with initial seed $(\hat{\curly{L}},\hat{B},\hat{L},\hat{G})$ that we will apply Theorem~\ref{t:rescaledQCA}.\label{page:fix-theta-end}

In Figure~\ref{fig:initialseedLoc36-homogenised}, we give this homogenised initial cluster for our running example with $k=3$ and $n=6$.

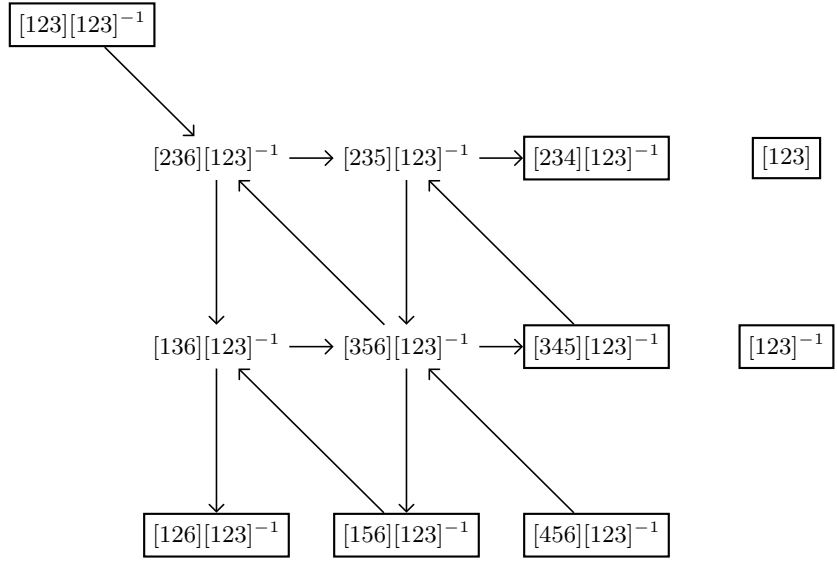
\begin{figure}[t]
\begin{center}
{\footnotesize
\scalebox{1}{\begin{tikzpicture}[node distance=2.5cm,on grid,>=angle 90]

\node (11) at (0,0) {$[236][123]^{-1}$}; 
\node (12) [right=of 11] {$[235][123]^{-1}$};
\node (13) [right=of 12,rectangle,draw=black,thick] {$[234][123]^{-1}$};

\node (21) [below=of 11] {$[136][123]^{-1}$}; 
\node (22) [right=of 21] {$[356][123]^{-1}$};
\node (23) [right=of 22,rectangle,draw=black,thick] {$[345][123]^{-1}$};

\node (31) [below=of 21,rectangle,draw=black,thick] {$[126][123]^{-1}$}; 
\node (32) [right=of 31,rectangle,draw=black,thick] {$[156][123]^{-1}$};
\node (33) [right=of 32,rectangle,draw=black,thick] {$[456][123]^{-1}$};

\node (00) [above left=of 11,rectangle,draw=black,thick] {$[123][123]^{-1}$};

\node (14) [right=of 13,rectangle,draw=black,thick] {$[123]$};
\node (24) [right=of 23,rectangle,draw=black,thick] {$[123]^{-1}$};

\draw[semithick,->] (11) to (12);
\draw[semithick,->] (12) to (13);

\draw[semithick,->] (21) to (22);
\draw[semithick,->] (22) to (23);

\draw[semithick,->] (11) to (21);
\draw[semithick,->] (21) to (31);

\draw[semithick,->] (12) to (22);
\draw[semithick,->] (22) to (32);

\draw[semithick,->] (22) to (11);
\draw[semithick,->] (23) to (12);

\draw[semithick,->] (32) to (21);
\draw[semithick,->] (33) to (22);

\draw[semithick,->] (00) to (11);

\end{tikzpicture}}
}
\end{center}
\caption{\label{fig:initialseedLoc36-homogenised}Homogenised initial cluster for a quantum cluster algebra structure on $\mathrm{Loc}(\KqGr{3}{6})$.}
\end{figure}

\begin{lemma} As in Proposition~\ref{p:dhom}, let $\sigma$ be the automorphism of $\KqMat{k}{n-k}$ defined by $\sigma(X_{ij})=qX_{ij}$.  Then there is an automorphism $\hat{\sigma}$ of $\KqMat{k}{n-k}[Y^{\pm 1}; \sigma]$ defined by $\hat{\sigma}|_{\KqMat{k}{n-k}}=\sigma$ and $\hat{\sigma}(Y)=Y$.
\end{lemma}

\begin{proof} One easily sees that $\hat{\sigma}$ respects the relations $YX_{ij}=\sigma(X_{ij})Y$ in the skew-Laurent extension.
\end{proof}

\begin{corollary}\label{c:automtau} There is an automorphism of $\mathrm{Loc}(\KqGr{k}{n-k})$ defined by $\tau=\alpha \circ \hat{\sigma} \circ \alpha^{-1}$, where $\alpha$ is the dehomogenisation isomorphism. \qed
\end{corollary}

Denote by $\hat{\underline{t}}$ the vector with $\hat{\underline{t}}_{(0,0)}=0$ and $\hat{\underline{t}}_{(r,s)}=G_{(r,s)}$.  Then the vector $\hat{\underline{t}}$ described a grading for $\hat{B}$ above, since $G$ was a grading for $B$.  The choice of $\hat{\underline{t}}_{(0,0)}=0$ is the unique one such that $\hat{\underline{t}}$ is indeed an extension of the grading $G$ to a grading for $\hat{B}$ but is also consistent with the natural degree of $\hat{L}_{(0,0)}=[1\dotsm k][1\dots k]^{-1}=1$ being zero.

Recall that the grading $G$ describes precisely the degree of the initial quantum cluster variables for the quantum cluster algebra structure on $\KqMat{k}{n-k}$, where ``degree'' means as a homogeneous polynomial in the matrix generators.  Indeed we saw that the quantum cluster algebra structure on $\KqMat{k}{n-k}$ is precisely graded by this natural grading.  Then it is clear from its definition that $\sigma$ acts by multiplication by $q$ to this degree, as we noted previously.  Passing this through the isomorphism $\alpha$, we see that $\tau$ is exactly the automorphism induced by $\hat{\underline{t}}$, or equivalently that $\hat{\underline{t}}$ may be recovered from $\tau$.

Now we apply Theorem~\ref{t:rescaledQCA}.  

\begin{proposition}\label{p:rescaledLoctilde} Let $\mathrm{Loc}(\KqGr{k}{n})$ have the graded quantum cluster algebra structure induced by the initial seed $(\hat{\curly{L}},\hat{B},\hat{L},\hat{G})$.  Then there exists a graded quantum cluster algebra structure on a subalgebra $\widetilde{\mathrm{Loc}}(\KqGr{k}{n})$ of the skew-Laurent extension $\mathrm{Loc}(\KqGr{k}{n})[Z^{\pm 1};\tau]$ with initial data
\begin{itemize}
\item $\tilde{\curly{L}}=\{ [1\dotsm k][1\dotsm k]^{-1}Z \} \union \{ \tilde{\curly{L}}_{(r,s)}=q^{\hat{\underline{t}}_{(r,s)}/2}\hat{\curly{L}}_{(r,s)}Z \mid 1\leq r\leq k,\ 1\leq s\leq n-k \} \union \{ q^{1/2}[1\dotsm k]Z^{-1}, q^{-1/2}[1\dotsm k]^{-1}Z \}$
\item $\tilde{B}=\hat{B}$,
\item $\tilde{L}$ satisfies $\tilde{L}_{(r_{1},s_{1}),(r_{s},s_{2})}=\hat{L}_{(r_{1},s_{1}),(r_{2},s_{2})}+\hat{\underline{t}}_{(r_{2},s_{2})}-\hat{\underline{t}}_{(r_{1},s_{1})}$, and
\item $\tilde{G}=0$.
\end{itemize}
Here the automorphism $\tau$ is as described in the previous lemma, inducing the grading $\hat{\underline{t}}$.
\end{proposition}

\begin{proof} We apply Theorem~\ref{t:rescaledQCA} with $\underline{t}=\hat{\underline{t}}$ and $\underline{u}=-\hat{G}$.  As noted above, both are gradings for $\hat{B}$.  Then one easily checks that $(\underline{t}\wedge \underline{u})_{ij}=t_{j}-t_{i}$, giving the form of $\tilde{L}$ as stated, and we have $\tilde{G}=\hat{G}+(-\hat{G})=0$.
\end{proof}

Indeed, applying Corollary~\ref{c:formofrescaledQCV} to this setting, we have that the quantum cluster variables for this quantum cluster algebra structure on $\widetilde{\mathrm{Loc}}(\KqGr{k}{n})$ are in bijection with those of $\mathrm{Loc}(\KqGr{k}{n})$ and furthermore the former have the form of a product of a power of $q$, a quantum cluster variable for $\mathrm{Loc}(\KqGr{k}{n})$ and a power of $Z$.

Thus for our running example $k=3$ and $n=6$ we have as initial cluster for the quantum cluster algebra $\widetilde{\mathrm{Loc}}(\KqGr{3}{6})$ that shown in Figure~\ref{fig:initialseedLoc36-homog-and-rescaled}.  The values of $\hat{\underline{t}}$ giving the powers of $q$ appearing are derived by reading off from Figure~\ref{fig:initialseed33} the degrees of the corresponding variables as homogeneous polynomials in the matrix generators.  So for example $\hat{\underline{t}}_{(2,2)}=2$ since $[356][123]^{-1}=\alpha(\minor{23}{23})$ and the latter has degree 2.

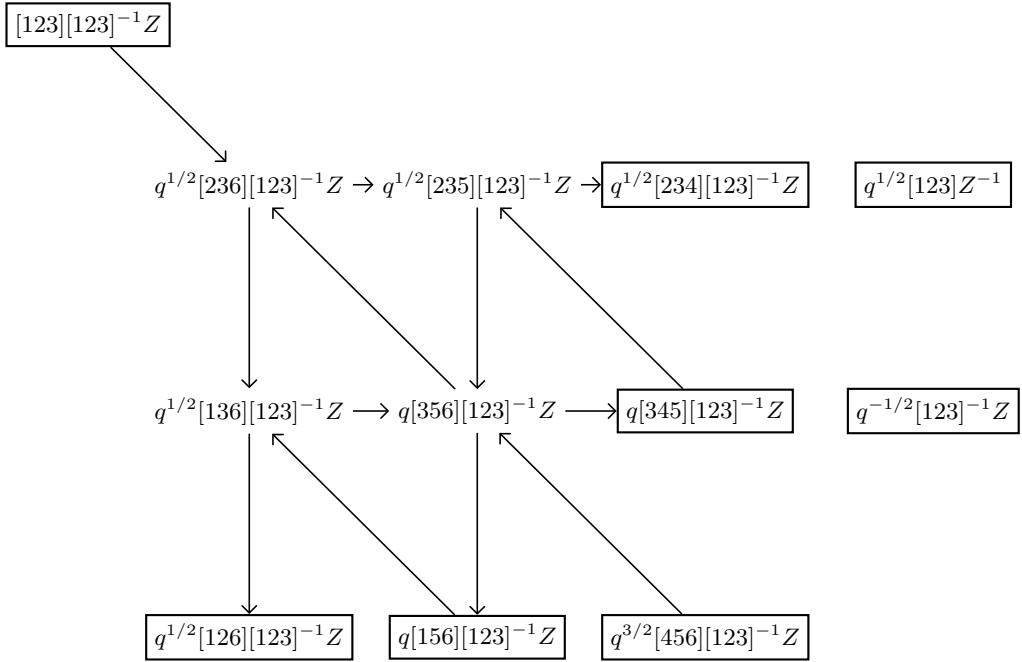
\begin{figure}[t]
\begin{center}
{\footnotesize
\scalebox{1}{\begin{tikzpicture}[node distance=3cm,on grid,>=angle 90]

\node (11) at (0,0) {$q^{1/2}[236][123]^{-1}Z$}; 
\node (12) [right=of 11] {$q^{1/2}[235][123]^{-1}Z$};
\node (13) [right=of 12,rectangle,draw=black,thick] {$q^{1/2}[234][123]^{-1}Z$};

\node (21) [below=of 11] {$q^{1/2}[136][123]^{-1}Z$}; 
\node (22) [right=of 21] {$q[356][123]^{-1}Z$};
\node (23) [right=of 22,rectangle,draw=black,thick] {$q[345][123]^{-1}Z$};

\node (31) [below=of 21,rectangle,draw=black,thick] {$q^{1/2}[126][123]^{-1}Z$}; 
\node (32) [right=of 31,rectangle,draw=black,thick] {$q[156][123]^{-1}Z$};
\node (33) [right=of 32,rectangle,draw=black,thick] {$q^{3/2}[456][123]^{-1}Z$};

\node (00) [above left=of 11,rectangle,draw=black,thick] {$[123][123]^{-1}Z$};

\node (14) [right=of 13,rectangle,draw=black,thick] {$q^{1/2}[123]Z^{-1}$};
\node (24) [right=of 23,rectangle,draw=black,thick] {$q^{-1/2}[123]^{-1}Z$};

\draw[semithick,->] (11) to (12);
\draw[semithick,->] (12) to (13);

\draw[semithick,->] (21) to (22);
\draw[semithick,->] (22) to (23);

\draw[semithick,->] (11) to (21);
\draw[semithick,->] (21) to (31);

\draw[semithick,->] (12) to (22);
\draw[semithick,->] (22) to (32);

\draw[semithick,->] (22) to (11);
\draw[semithick,->] (23) to (12);

\draw[semithick,->] (32) to (21);
\draw[semithick,->] (33) to (22);

\draw[semithick,->] (00) to (11);

\end{tikzpicture}}
}
\end{center}
\caption{\label{fig:initialseedLoc36-homog-and-rescaled}Initial cluster for a quantum cluster algebra structure on $\widetilde{\mathrm{Loc}}(\KqGr{3}{6})$.}
\end{figure}

\begin{lemma}\label{l:centralelt} The element $q^{-1/2}[1\dotsm k]^{-1}Z$ is central in $\mathrm{Loc}(\KqGr{k}{n})[Z^{\pm 1};\tau]$ and hence is a central coefficient in the quantum cluster algebra $\widetilde{\mathrm{Loc}}(\KqGr{k}{n})$.
\end{lemma}

\begin{proof} We chose $\tau$ so that 
\[ \tau([1\dotsm k])=\alpha(\sigma(\alpha^{-1}([1\dotsm k])))=\alpha(\sigma(Y))=\alpha(Y)=[1\dotsm k], \] and hence $Z$ commutes with $[1\dotsm k]$, and so that $\alpha(Y)=[1\dotsm k]$ and $Z$ satisfy the same quasi-commutation relations with $\alpha(\KqMat{k}{n-k})$.  It follows that the stated element is central.
\end{proof}

\begin{corollary}\label{c:quotisQCA} The quotient algebra $\widetilde{\mathrm{Loc}}(\KqGr{k}{n})/(q^{-1/2}[1\dotsm k]^{-1}Z-1)$ inherits a graded quantum cluster algebra structure.
\end{corollary}

\begin{proof} It is straightforward to see that the quotient of a quantum cluster algebra by a central coefficient is again a quantum cluster algebra, with the natural quotient data.  Since the element $q^{-1/2}[1\dotsm k]^{-1}Z$ has degree 0 (as $\tilde{G}=0$), this quantum cluster algebra is again graded by the grading $-\hat{G}$, suitably restricted; indeed we see that this grading is equal to $\underline{1}=(1,\dotsc ,1)\in \integ^{k(n-k)+1}$.
\end{proof}

In particular, in this quotient the two coefficients $q^{1/2}[1\dotsm k]Z^{-1}$ and $q^{-1/2}[1\dotsm k]^{-1}Z$ are both identified with the identity and as such may be deleted from the quantum cluster algebra data with no effect, which we do.

\begin{theorem}\label{t:quotisotoGr} The quotient algebra $\widetilde{\mathrm{Loc}}(\KqGr{k}{n})/(q^{-1/2}[1\dotsm k]^{-1}Z-1)$ is isomorphic to $\KqGr{k}{n}$.  Hence the quantum Grassmannian $\KqGr{k}{n}$ admits a graded quantum algebra structure.
\end{theorem}

\begin{proof} We proceed in two steps.  First, we show the existence of a surjective homomorphism from $\KqGr{k}{n}$ to the quotient $\widetilde{\mathrm{Loc}}(\KqGr{k}{n})/(q^{-1/2}[1\dotsm k]^{-1}Z-1)$.  Then we show that the latter has the same Gel\cprime fand--Kirillov dimension as $\KqGr{k}{n}$, from which it follows that the two algebras are isomorphic.

Let us denote by $z$ the (central) element $q^{-1/2}[1\dotsm k]^{-1}Z \in \mathrm{Loc}(\KqGr{k}{n})[Z^{\pm 1};\tau]$.  Then let $f\colon \KqGr{k}{n} \to \mathrm{Loc}(\KqGr{k}{n})[Z^{\pm 1};\tau]$ be the linear map defined on the generating quantum Pl\"{u}cker coordinates $[I]$ of $\KqGr{k}{n}$ by $f([I])=[I]z$, extended to first to monomials multiplicatively and then extended linearly.  Notice that $f$ is a map whose codomain is the whole skew-Laurent extension defined in Proposition~\ref{p:rescaledLoctilde}.  Since $\KqGr{k}{n}$ is spanned as a vector space by monomials in the quantum Pl\"{u}cker coordinates and since the defining (generalised) quantum Pl\"{u}cker relations in $\KqGr{k}{n}$ are homogeneous (\cite[Remark~3.3]{Kolb}), the centrality of $z$ implies that this yields a well-defined algebra homomorphism.

Next we establish that the image of $f$ lies in $\widetilde{\mathrm{Loc}}(\KqGr{k}{n})$.  This follows from Theorem~\ref{t:thetaisscalingpower}: we defined the elements $\curly{L}_{(r,s)}$ to be the images under the dehomogenisation isomorphism $\alpha$ of the initial cluster variables for the quantum cluster algebra structure on $\KqMat{k}{n-k}$.  In particular, as noted after Theorem~\ref{t:thetaisscalingpower}, every quantum cluster variable of $\widetilde{\mathrm{Loc}}(\KqGr{k}{n})$ is equal to an element of $\KqGr{k}{n}$ multiplied by some power of $z$.  For this is exactly why the choices in Proposition~\ref{p:rescaledLoctilde} were made: Theorem~\ref{t:thetaisscalingpower} gives that every quantum cluster variable $v$ in $\mathrm{Loc}(\KqGr{k}{n})$ has a unique expression as an element of $\KqGr{k}{n}$ multiplied by $([1\dotsc k]^{-1})^{\theta(\alpha^{-1}(v))}$ and we applied the re-scaling construction of Theorem~\ref{t:rescaledQCA} using the data from $\theta$ (now ``fixed'' to be a genuine grading), re-scaling every quantum cluster variable precisely by the power of $Z$ needed to ensure that we obtain as quantum cluster variables elements of $\KqGr{k}{n}$ multiplied by powers of $z$.

Furthermore, since the quantum cluster variables by definition generate $\widetilde{\mathrm{Loc}}(\KqGr{k}{n})$, $f$ is a surjective homomorphism onto $\widetilde{\mathrm{Loc}}(\KqGr{k}{n})$.  Composing $f$ with the natural projection of $\widetilde{\mathrm{Loc}}(\KqGr{k}{n})$ onto the quotient by the ideal generated by $z-1$, we have a surjective homomorphism $g \colon \KqGr{k}{n} \to \widetilde{\mathrm{Loc}}(\KqGr{k}{n})/(z-1)$, as we wanted.

Next, we show the equality of the Gel\cprime fand--Kirillov dimensions of the domain and codomain of $g$.  The GK-dimension of $\KqGr{k}{n}$ is well-known to be $k(n-k)+1$ so we compute the GK-dimension of the quotient that is the codomain.  By construction, $\widetilde{\mathrm{Loc}}(\KqGr{k}{n})$ contains the quantum affine space $\mathbb{A}_{q}$ whose generators are precisely the elements in the initial seed of Proposition~\ref{p:rescaledLoctilde} except $z^{-1}$.  Since $\tilde{L}\setminus \{z^{-1} \}$ is a quasi-commuting set, being a quantum cluster, this is clear.

Now $\mathbb{A}_q$ has $k(n-k)+2$ generators, so $k(n-k)+2=\mathrm{GKdim}\ \mathbb{A}_q \leq \mathrm{GKdim} \widetilde{\mathrm{Loc}}(\KqGr{k}{n})$.  On the other hand, by the quantum Laurent phenomenon (\cite[Corollary~5.2]{BZ-QCA}), $\widetilde{\mathrm{Loc}}(\KqGr{k}{n})$ is contained in the quantum torus $\mathbb{T}_q$ associated to $\mathbb{A}_q$. So $\mathrm{GKdim}\ \widetilde{\mathrm{Loc}}(\KqGr{k}{n}) \leq \mathrm{GKdim}\ \mathbb{T}_q =k(n-k)+2$.

Hence, $\mathrm{GKdim}\ \widetilde{\mathrm{Loc}}(\KqGr{k}{n})=k(n-k)+2$ and so it follows from \cite[Proposition~3.15]{KrauseLenagan} that $\mathrm{GKdim}\ \widetilde{\mathrm{Loc}}(\KqGr{k}{n})/(z-1)\leq (k(n-k)+2)-1=k(n-k)+1$ since $z-1$ is certainly regular.

On the other hand,  $\widetilde{\mathrm{Loc}}(\KqGr{k}{n})/(z-1)$ contains (a copy of) the quantum affine space generated by the images of the indeterminates in the initial seed $\tilde{\curly{L}}$ of $\widetilde{\mathrm{Loc}}(\KqGr{k}{n})$ except $z$ and $z^{-1}$.  So $\mathrm{GKdim}\ \widetilde{\mathrm{Loc}}(\KqGr{k}{n})/(z-1) \geq k(n-k)+1$ and hence we in fact have equality.

Now since $\KqGr{k}{n}$ and $\widetilde{\mathrm{Loc}}(\KqGr{k}{n})$ are domains, it follows that these are isomorphic, as every epimorphism of domains of the same Gel\cprime fand--Kirillov dimension is an isomorphism (\cite[Proposition~3.15]{KrauseLenagan}, as previously).
\end{proof}

In Figure~\ref{fig:initialseedKqGr36} we show the initial data for the quantum cluster algebra structure thus obtained on $\KqGr{3}{6}$.  We note that due to the powers of $q$ that are present, this is not identical to the quantum cluster algebra structure obtained in the authors' earlier work (\cite{Gr2nSchubertQCA}).  However, to conclude, we may apply Proposition~\ref{p:rescaleQCAbypowersofq}, for we see that the power of $q$ appearing in the expression for the $(r,s)$ variable is exactly $q^{\frac{\hat{\underline{t}}_{(r,s)}+1}{2}}$.  The corresponding vector $\transpose{(\hat{\underline{t}}_{(0,0)}+1,\dotsc ,\hat{\underline{t}}_{(k,n-k)}+1)}$ is a grading for the relevant exchange matrix, since both $\hat{\underline{t}}$ and $(1,\dotsc ,1)$ are.  Then we may apply Proposition~\ref{p:rescaleQCAbypowersofq} with the negative of this grading, to obtain an isomorphic quantum cluster algebra without these powers of $q$.  In this way, we recover exactly the quantum cluster algebra structure obtained in \cite{Gr2nSchubertQCA}.  Figure~\ref{fig:finalQCAstrKqGr36} shows this final re-scaled quantum cluster algebra structure on $\KqGr{3}{6}$.

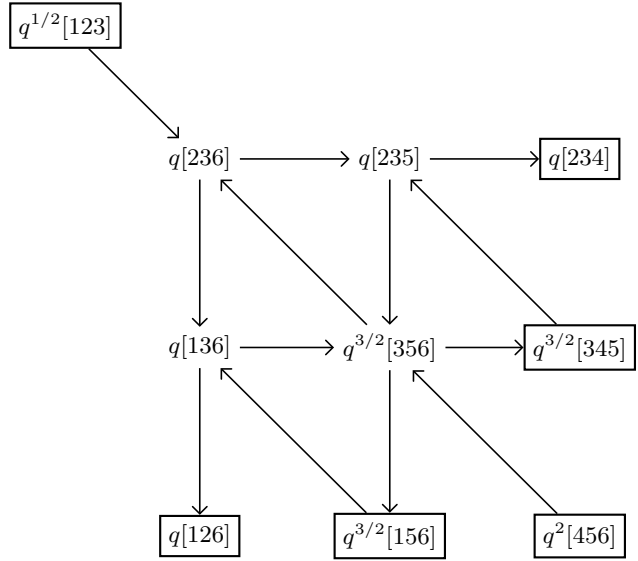
\begin{figure}[t]
\begin{center}
{\footnotesize
\scalebox{1}{\begin{tikzpicture}[node distance=2.5cm,on grid,>=angle 90]

\node (11) at (0,0) {$q[236]$}; 
\node (12) [right=of 11] {$q[235]$};
\node (13) [right=of 12,rectangle,draw=black,thick] {$q[234]$};

\node (21) [below=of 11] {$q[136]$}; 
\node (22) [right=of 21] {$q^{3/2}[356]$};
\node (23) [right=of 22,rectangle,draw=black,thick] {$q^{3/2}[345]$};

\node (31) [below=of 21,rectangle,draw=black,thick] {$q[126]$}; 
\node (32) [right=of 31,rectangle,draw=black,thick] {$q^{3/2}[156]$};
\node (33) [right=of 32,rectangle,draw=black,thick] {$q^{2}[456]$};

\node (00) [above left=of 11,rectangle,draw=black,thick] {$q^{1/2}[123]$};

\draw[semithick,->] (11) to (12);
\draw[semithick,->] (12) to (13);

\draw[semithick,->] (21) to (22);
\draw[semithick,->] (22) to (23);

\draw[semithick,->] (11) to (21);
\draw[semithick,->] (21) to (31);

\draw[semithick,->] (12) to (22);
\draw[semithick,->] (22) to (32);

\draw[semithick,->] (22) to (11);
\draw[semithick,->] (23) to (12);

\draw[semithick,->] (32) to (21);
\draw[semithick,->] (33) to (22);

\draw[semithick,->] (00) to (11);

\end{tikzpicture}}
}
\end{center}
\caption{\label{fig:initialseedKqGr36}Initial cluster for a quantum cluster algebra structure on $\KqGr{3}{6}$, prior to a final re-scaling.}
\end{figure}

\begin{figure}[t]
\begin{center}
{\footnotesize
\scalebox{1}{\begin{tikzpicture}[node distance=2.5cm,on grid,>=angle 90]

\node (11) at (0,0) {$[236]$}; 
\node (12) [right=of 11] {$[235]$};
\node (13) [right=of 12,rectangle,draw=black,thick] {$[234]$};

\node (21) [below=of 11] {$[136]$}; 
\node (22) [right=of 21] {$[356]$};
\node (23) [right=of 22,rectangle,draw=black,thick] {$[345]$};

\node (31) [below=of 21,rectangle,draw=black,thick] {$[126]$}; 
\node (32) [right=of 31,rectangle,draw=black,thick] {$[156]$};
\node (33) [right=of 32,rectangle,draw=black,thick] {$[456]$};

\node (00) [above left=of 11,rectangle,draw=black,thick] {$[123]$};

\draw[semithick,->] (11) to (12);
\draw[semithick,->] (12) to (13);

\draw[semithick,->] (21) to (22);
\draw[semithick,->] (22) to (23);

\draw[semithick,->] (11) to (21);
\draw[semithick,->] (21) to (31);

\draw[semithick,->] (12) to (22);
\draw[semithick,->] (22) to (32);

\draw[semithick,->] (22) to (11);
\draw[semithick,->] (23) to (12);

\draw[semithick,->] (32) to (21);
\draw[semithick,->] (33) to (22);

\draw[semithick,->] (00) to (11);

\end{tikzpicture}}
}
\end{center}
\caption{\label{fig:finalQCAstrKqGr36}Initial cluster for a quantum cluster algebra structure on $\KqGr{3}{6}$, following a final re-scaling.}
\end{figure}
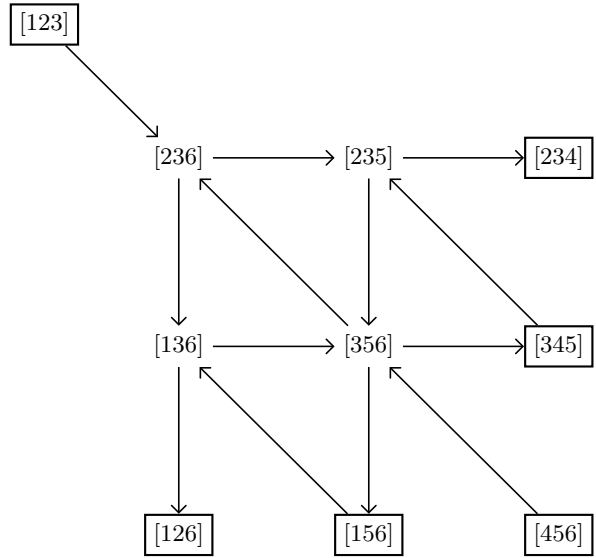

\begin{remark} The preceding Corollary establishes that $\KqGr{k}{n}$ is a \emph{graded} quantum cluster algebra and we have seen that this grading is the standard grading on $\KqGr{k}{n}$, with the quantum Pl\"{u}cker coordinates in degree one.  Then in particular it follows from the general theory of graded quantum cluster algebras that every quantum cluster variable is homogeneous with respect to this graded, a phenomenon observed in the authors' earlier work (\cite{Gr2nSchubertQCA}) (and only experimentally for $\KqGr{3}{n}$, $n=6,7,8$).
\end{remark}

\begin{remark} We note that we have worked throughout over the field $\mathbb{Q}(q)$, that is, with $q$ transcendental over $\mathbb{Q}$.  This assumption is necessary because this is the context in which the main theorem of Gei\ss--Leclerc--Schr\"{o}er is proved (\cite[Theorem~12.3]{GLS-QuantumPFV}), for a number of technical reasons.  Since that theorem provides the starting point for our lifting, namely the quantum cluster algebra structure on quantum matrices, we must make this assumption too.  However our methods here only use that $q$ is not a root of unity, so that if the aforementioned result is extended, our conclusion will also follow immediately without need for modification. 

So, in line with Conjecture~12.7 of \cite{GLS-QuantumPFV}, we conjecture that the above quantum cluster algebra structure on the quantum Grassmannian can be realised on an integral form, i.e.\ over $\mathbb{Q}[q,q^{-1}]$.  Indeed the explicit descriptions of the quantum cluster variables in the authors' earlier work suggest that this structure may even be defined over $\integ[q,q^{-1}]$.  However, many of the constructions we have applied, notably Theorem~\ref{t:rescaledQCA}, involve powers of $q^{1/2}$ and it appears to be a delicate matter to see that $q^{1/2}$ does not enter into the final quantum cluster algebra structure.
\end{remark}

\begin{remark} As noted in the introduction, we expect that the methods presented here---or generalisations of them---can be used to establish the existence of (graded) quantum cluster algebra structures on the quantized coordinate rings of arbitrary partial flag varieties.  The relationships between the latter and their localisations that give the quantized coordinate rings of the big cells of the corresponding partial flag variety are well-understood and dehomogenisation isomorphisms such as that used here are known.  Modifications of the constructions here may be necessary, however.  For example, multi-gradings may be needed where the coordinate rings of the big cells involve localisation at several elements.
\end{remark}

\clearpage

\pagebreak

\section{Appendix}\label{appendix}

Here we gather the aforementioned details of the calculations used in the proof of Theorem~\ref{t:rescaledQCA}.  Notation is as in that proof.

\begin{lemma} $\tilde{X}_{i}^{a_{i}}=q^{\frac{1}{2}a_{i}^{2}t_{i}u_{i}}X_{i}^{a_{i}}z^{a_{i}u_{i}}$
\end{lemma}

\begin{proof} If $a_{i}\geq 0$ then
\begin{align*}
\tilde{X}_{i}^{a_{i}} & = \left(q^{\frac{t_{i}u_{i}}{2}}X_{i}z^{u_{i}}\right)^{a_{i}} \\
 & = q^{\frac{a_{i}t_{i}u_{i}}{2}}\left( X_{i}z^{u_{i}}\right)^{a_{i}} \\
 & = q^{\frac{a_{i}t_{i}u_{i}}{2}}q^{\left(\sum_{j=1}^{a_{i}-1} j\right)t_{i}u_{i}}X_{i}^{a_{i}}z^{a_{i}u_{i}} \\
 & = q^{\frac{a_{i}t_{i}u_{i}}{2}}q^{\left(\frac{a_{i}(a_{i}-1)}{2}\right)t_{i}u_{i}}X_{i}^{a_{i}}z^{a_{i}u_{i}} \\
 & = q^{\frac{1}{2}a_{i}^{2}t_{i}u_{i}}X_{i}^{a_{i}}z^{a_{i}u_{i}} 
\end{align*}
since $z^{u_{i}}X_{i}=q^{t_{i}u_{i}}X_{i}z^{u_{i}}$.  The sum $\left(\sum_{j=1}^{a_{i}-1} j\right)t_{i}u_{i}$ arises from moving $z^{u_{i}}$'s to the right the required number of times to rearrange the product as shown.  It is straightforward to check that the claim is also correct for $a_{i}\leq 0$ by similar means.
\end{proof}

Set $\beta_{i}=q^{\frac{1}{2}a_{i}^{2}t_{i}u_{i}}$.

\begin{lemma} $z^{a_{i}u_{i}}\left(\prod_{j=i+1}^{r} X_{j}^{a_{j}}z^{a_{j}u_{j}}\right)=\left(\prod_{j=i+1}^{r} q^{a_{i}a_{j}t_{j}u_{i}}\right)\left(\prod_{j=i+1}^{r} X_{j}^{a_{j}}z^{a_{j}u_{j}}\right)z^{a_{i}u_{i}}$
\end{lemma}

\begin{proof} This follows from the defining quasi-commutation relation $zX_{j}=q^{t_{j}}X_{j}z$ and noting that $z$ commutes with itself.
\end{proof}

Set $\alpha_{i}=\prod_{j=i+1}^{r} q^{a_{i}a_{j}t_{j}u_{i}}=\prod_{i<j} q^{a_{i}a_{j}t_{j}u_{i}}$ so that
\[ z^{u_{i}a_{i}}\prod_{i<j} X_{j}^{a_{j}}z^{a_{j}u_{j}}=\alpha_{i}\left(\prod_{i<j} X_{j}^{a_{j}}z^{a_{j}u_{j}} \right)z^{u_{i}a_{i}}, \]
 by the preceding lemma.

\begin{lemma} $\prod_{i=1}^{r} \tilde{X}_{i}^{a_{i}} = \left(\prod_{i=1}^{r-1} \alpha_{i}\right)\left( \prod_{i=1}^{r} \beta_{i} \right)X_{1}^{a_{1}}\dotsm X_{r}^{a_{r}}z^{\left(\sum_{i=1}^{r} a_{i}u_{i}\right)}$
\end{lemma}

\begin{proof}
\begin{align*}
\prod_{i=1}^{r} \tilde{X}_{i}^{a_{i}} & = \prod_{i=1}^{r} \beta_{i}X_{i}^{a_{i}}z^{a_{i}u_{i}} \\
 & = \left(\prod_{i=1}^{r} \beta_{i}\right)\prod_{i=1}^{r} X_{i}^{a_{i}}z^{a_{i}u_{i}} \\
 & = \left(\prod_{i=1}^{r-1} \alpha_{i}\right)\left( \prod_{i=1}^{r} \beta_{i} \right)X_{1}^{a_{1}}\dotsm X_{r}^{a_{r}}z^{\left(\sum_{i=1}^{r} a_{i}u_{i}\right)}
\end{align*}
by using the above lemmas repeatedly.
\end{proof}

\pagebreak
\begin{lemma} {\ }
\begin{enumerate}[label=(\alph*)]
\item $\prod_{i=1}^{r-1} \alpha_{i}=q^{\sum_{i<j} a_{i}a_{j}t_{j}u_{i}}$
\item $\prod_{i=1}^{r} \beta_{i} = q^{\frac{1}{2}\sum_{i=1}^{r} a_{i}^{2}t_{i}u_{i}}$
\end{enumerate}
\end{lemma}

\begin{proof} These equalities are immediate from the definitions of $\alpha_{i}$ and $\beta_{i}$ respectively.
\end{proof}

\begin{proposition} $\tilde{M}(a_{1},\dotsc ,a_{r})=q^{\frac{1}{2}\sum_{i=1}^{r} \sum_{j=1}^{r} a_{i}a_{j}t_{i}u_{j}}M(a_{1},\dotsc ,a_{r})z^{\left(\sum_{i=1}^{r} a_{i}u_{i}\right)}$
\end{proposition}

\begin{proof}
\begin{align*}
\tilde{M}(a_{1},\dotsc ,a_{r}) & = q^{\frac{1}{2}\sum_{i<j} a_{i}a_{j}\tilde{l}_{ji}}\left(\prod_{i=1}^{r} \tilde{X}_{i}^{a_{i}} \right) \\
 & = q^{\frac{1}{2}\sum_{i<j} a_{i}a_{j}(l_{ji}+t_{i}u_{j}-t_{j}u_{i})}\left(\prod_{i=1}^{r} \tilde{X}_{i}^{a_{i}} \right) \\
 & = q^{\frac{1}{2}\sum_{i<j} a_{i}a_{j}(l_{ji}+t_{i}u_{j}-t_{j}u_{i})}\left(\prod_{i=1}^{r-1} \alpha_{i}\right)\left( \prod_{i=1}^{r} \beta_{i} \right)X_{1}^{a_{1}}\dotsm X_{r}^{a_{r}}z^{\left(\sum_{i=1}^{r} a_{i}u_{i}\right)} \\
 & = q^{\frac{1}{2}\sum_{i<j} a_{i}a_{j}(l_{ji}+t_{i}u_{j}-t_{j}u_{i})}q^{\sum_{i<j} a_{i}a_{j}t_{j}u_{i}}q^{\frac{1}{2}\sum_{i=1}^{r} a_{i}^{2}t_{i}u_{i}}X_{1}^{a_{1}}\dotsm X_{r}^{a_{r}}z^{\left(\sum_{i=1}^{r} a_{i}u_{i}\right)} \\
 & = q^{\frac{1}{2}\sum_{i<j} a_{i}a_{j}(l_{ji}+t_{i}u_{j}-t_{j}u_{i})}q^{\sum_{i<j} a_{i}a_{j}t_{j}u_{i}}q^{\frac{1}{2}\sum_{i=1}^{r} a_{i}^{2}t_{i}u_{i}}q^{-\frac{1}{2}\sum_{i<j} a_{i}a_{j}l_{ji}}M(a_{1},\dotsc ,a_{r})z^{\left(\sum_{i=1}^{r} a_{i}u_{i}\right)} \\
 & = q^{\frac{1}{2}\left((\sum_{i<j} a_{i}a_{j}(t_{i}u_{j}+t_{j}u_{i}))+(\sum_{i=1}^{r} a_{i}^{2}t_{i}u_{i})\right)}M(a_{1},\dotsc ,a_{r})z^{\left(\sum_{i=1}^{r} a_{i}u_{i}\right)} \\
 & = q^{\frac{1}{2}\sum_{i=1}^{r} \sum_{j=1}^{r} a_{i}a_{j}t_{i}u_{j}}M(a_{1},\dotsc ,a_{r})z^{\left(\sum_{i=1}^{r} a_{i}u_{i}\right)}
\end{align*}
\end{proof}

\noindent This is the equality as claimed in the proof of Theorem~\ref{t:rescaledQCA}.

\small

\bibliographystyle{amsplain}
\bibliography{biblio}\label{references}

\normalsize

\end{document}